\documentclass[12pt,reqno]{amsart}
\usepackage{amsmath,amsfonts,amsthm,amsopn,amssymb}
\usepackage{cite,marginnote}
\usepackage{color,enumitem,graphicx}
\usepackage{cases}
\usepackage[colorlinks=true,urlcolor=blue,citecolor=red,linkcolor=blue,
linktocpage,pdfpagelabels,bookmarksnumbered,bookmarksopen]{hyperref}
\usepackage[english]{babel}

\usepackage[left=2.9cm,right=2.9cm,top=2.8cm,bottom=2.8cm]{geometry}
\usepackage[hyperpageref]{backref}

\allowdisplaybreaks
\numberwithin{equation}{section}

\makeindex
\makeindex
\newenvironment{Abstract}{\textbf{Abstract}\mbox{  }}{ }
\newenvironment{key words}{\textbf{Keywords}\mbox{  }}{ }

\newtheorem{theorem}{Theorem}[section]
\newtheorem{definition}[theorem]{Definition}
\newtheorem{lemma}[theorem]{Lemma}
\newtheorem{corollary}[theorem]{Corollary}
\newtheorem{proposition}[theorem]{Proposition}
\renewenvironment{proof}{\noindent{\textbf{Proof.}}}{\hfill$\Box$}
\newtheorem{remark}[theorem]{Remark}

\newcommand{\ud}{\mathrm{d}}

\newcommand{\R}{\mathbb R}

\newcommand{\lab}{\label}
\newcommand{\bt}{\begin{theorem}}
\newcommand{\et}{\end{theorem}}
\newcommand{\bl}{\begin{lemma}}
\newcommand{\el}{\end{lemma}}
\newcommand{\bd}{\begin{definition}}
\newcommand{\ed}{\end{definition}}
\newcommand{\bc}{\begin{corollary}}
\newcommand{\ec}{\end{corollary}}
\newcommand{\bp}{\begin{proof}}
\newcommand{\ep}{\end{proof}}
\newcommand{\bx}{\begin{example}}
\newcommand{\ex}{\end{example}}
\newcommand{\bi}{\begin{exercise}}
\newcommand{\ei}{\end{exercise}}
\newcommand{\bo}{\begin{proposition}}
\newcommand{\eo}{\end{proposition}}
\newcommand{\br}{\begin{remark}}
\newcommand{\er}{\end{remark}}
\newcommand{\beq}{\begin{equation}}
\newcommand{\eeq}{\end{equation}}
\newcommand{\ba}{\begin{align}}
\newcommand{\ea}{\end{align}}
\newcommand{\bn}{\begin{enumerate}}
\newcommand{\en}{\end{enumerate}}
\newcommand{\bg}{\begin{align}}
\newcommand{\bcs}{\begin{cases}}
\newcommand{\ecs}{\end{cases}}

\newcommand{\bean}{\begin{eqnarray*}}
\newcommand{\eean}{\end{eqnarray*}}

\def\N{\mathbb{N}}

\def\R{\mathbb{R}}

\def\bd{\mathrm{bd}\,}

\begin{document}
\title[The mass-mixed case for NLS in dimension two]{The mass-mixed case for normalized solutions to NLS equations in dimension two}
\author[D.~Cassani]{Daniele Cassani$^*$}
\author[L.~Huang]{Ling Huang}
\author[C.~Tarsi]{Cristina Tarsi}
\author[X.~X.~Zhong]{Xuexiu Zhong$^{**}$}
\thanks{$^*$ Corresponding author: \texttt{daniele.cassani@uninsubria.it}}
\address[D.~Cassani]{\newline\indent Dip. di Scienza e Alta Tecnologia
\newline\indent
Universit\` a degli Studi dell'Insubria \&
\newline\indent
RISM-Riemann International School of Mathematics
\newline\indent
villa Toeplitz, Via G.B. Vico, 46-21100 Varese, Italy}
\email{\href{mailto:daniele.cassani@uninsubria.it}{daniele.cassani@uninsubria.it}}

\address[L.~Huang]{\newline\indent School of Mathematical Sciences
\newline\indent
South China Normal University
\newline\indent
Guangzhou 510631, PR China}
\email{\href{mailto:996987952@qq.com}{996987952@qq.com}}

\address[C.~Tarsi]{\newline\indent Dipartimento di Matematica
\newline\indent
Universit\` a degli Studi di Milano
\newline\indent
Via C. Saldini, 50-20133 Milano, Italy}
\email{\href{mailto:cristina.tarsi@unimi.it}{cristina.tarsi@unimi.it}}

\address[X.~X.~Zhong]{\newline\indent South China Research Center for Applied Mathematics and Interdisciplinary Studies
\newline\indent
\&  School of Mathematical Sciences, South China Normal University
\newline\indent
Guangzhou 510631, PR China}
\email{\href{mailto:zhongxuexiu1989@163.com}{zhongxuexiu1989@163.com}}

\thanks{$^{**}$Xuexiu Zhong was supported by the NSFC (No.12271184), Guangzhou Basic and Applied Basic Research Foundation(2024A04J10001).}

\maketitle

\noindent
{\small
\begin{Abstract}
\noindent We are concerned with positive normalized solutions $(u,\lambda)\in H^1(\mathbb{R}^2)\times\mathbb{R}$ to the following semi-linear Schr\"{o}dinger equations
$$
-\Delta u+\lambda u=f(u), \quad\text{in}~\mathbb{R}^2,
$$
satisfying the mass constraint $\int_{\R^2}|u|^2\mathrm{d}x=c^2$. We are interested in the so-called mass mixed case in which $f$ has $L^2$-subcritical growth at zero and critical growth at infinity, which in dimension two turns out to be of exponential rate. Under mild conditions, we establish the existence of two positive normalized solutions provided the prescribed mass is sufficiently small: one is a local minimizer and the second one is of mountain pass type. We also investigate the asymptotic behavior of solutions approaching the zero mass case, namely when $c\to 0^+$.
\end{Abstract}\\
\begin{key words}
Nonlinear Schr\"{o}dinger equation, Exponential critical growth, Positive normalized solution, General nonlinearities, Trudinger-Moser inequality, Asymptotic behavior.
\end{key words}
}

\indent

\section{Introduction and main results}
\noindent We consider the following problem
\begin{numcases}{}
&$-\Delta u+\lambda u =f(u), \quad u\in H^1(\R^2),\quad u>0$ \label{20230228e1}\\
&$\int_{\R^2}|u|^2\mathrm{d}x =c^2$ \label{20230228e2}
\end{numcases}
where $\lambda\in\R$ is unknown, appearing as a Lagrange multiplier and related to the value $c$ in condition \eqref{20230228e2} which represents the conservation of mass in the Non-Linear Schr\"odinger (NLS for brevity) equation. On the one hand, this problem in dimension $N\geq 3$ has recently attracted attention due to the connection with Physics and in particular the stability properties of solutions in the time dependent Schr\"odinger equation. On the other hand, the mathematical challenge has roots in the Gagliardo-Nirenberg inequality which gives rise to the so-called $L^2$-critical exponent $\bar{p}=2+4/N$. Indeed, for $p\in(2,\bar{p})\cup(\bar{p},2^*)$, problem \eqref{20230228e1}-\eqref{20230228e2} with $f(u)=u^{p-1}$ has exactly one solution for any $c>0$, whence for $p=\bar{p}$ a one parameter family of solutions shows up provided $c$ is explicitly given in terms of the best constant in the following Gagliardo-Nirenberg inequality:
$$\|u\|_{\bar{p}}^{\bar{p}}\leq C_{N,p}\|\nabla u\|_2^2\|u\|_2^{\frac{4}{N}}\ .$$
This phenomenon resembles what happens for \eqref{20230228e1} in the Sobolev critical case, namely when $f(u)=u^{2^*-1}$, and this is due to the appearance of the invariance under the group action $u_{\lambda}(x)=\lambda^{N/2}u(\lambda x)$. One has $L^2$-subcritical (equivalently mass subcritical), critical or supercritical growth, accordingly with $p<\bar{p}$, $p=\bar{p}$ or $p>\bar{p}$. Let us refer to \cite{Bartsch2016,Cazenave1982,Cid2001,Colin2010} and references therein. In particular, in the spirit of Brezis-Nirenberg, it is interesting to understand what happens in presence of critical growth when lower order perturbations are taken into account, see \cite{Soave2020,Soave2020a}. In dimension two, the notion of criticality goes beyond any power like growth up to the exponential. In this paper we aim at considering the case of nonlinearities which have mass subcritical growth near zero and critical exponential growth at infinity in the sense of \cite{Figueiredo1995} (see $(\mathbf{f_2})$ below), the so-called mass-mixed case.

\noindent Under suitable assumptions on $f$, solutions to \eqref{20230228e1}-\eqref{20230228e2} can be obtained as critical points of the $C^1$-functional $I:H^1(\R^2)\to\R$ defined by
$$
I(u):=\frac{1}{2}\|\nabla u\|_{2}^2-\int_{\R^2}F(u)\mathrm{d}x,
$$
where $F$ is the primitive function of $f$ vanishing at zero, constrained on the sphere
\begin{align*}
S_c:=\{u \in H^1(\mathbb{R}^2):\|u\|_2^2=c^2\}.
\end{align*}

\noindent Since we are interested in positive solutions, let $f(s)=0$ for $s\leq 0$ and let us state the following assumptions:
\begin{itemize}
\item[$\mathbf{\mathbf{(f_1)}}$] $f(0)=0,~f(s)>0$ for $s>0$ and there exists some $p\in (2,4)$ such that $$0<\liminf_{s\to 0^+}\frac{f(s)}{s^{p-1}}\leq \limsup_{s\to 0^+}\frac{f(s)}{s^{p-1}}<+\infty\ ;$$

\item[$\mathbf{\mathbf{(f_2)}}$] There exists $\alpha_0>0$ such that
$$
\lim\limits_{s\to +\infty}\frac{f(s)}{e^{\alpha s^2}}=0,~~\text{for}\ \alpha>\alpha_0
$$
and
$$
\lim\limits_{s\to +\infty}\frac{f(s)}{e^{\alpha s^2}}=+\infty,~~\text{for}\ \alpha<\alpha_0\ ;
$$

\item[$\mathbf{\mathbf{(f_3)}}$] $\lim\limits_{s\to\infty}\frac{F(s)}{f(s)s}=0$;

\item[$\mathbf{\mathbf{(f_4)}}$] There exists $\beta_0>0,$ such that $\liminf\limits_{s \rightarrow+\infty} \frac{f(s)s}{e^{\alpha_0s^2}}\geq\beta_0$;

\item[$\mathbf{(f_5)}$] $f(s)$ is non-decreasing for $s>0$.
\end{itemize}
\subsection{Main results}
\begin{theorem}\label{theorem 1.1}
Under the assumptions $\mathbf{(f_1)}$ and $\mathbf{(f_2)}$, there exists $c_0>0$ such that for any $c \in(0, c_0)$, $\left.I\right|_{S_{c}}$ has a local minimizer $u_c$, which is a positive, radially symmetric decreasing function and there exists $\lambda_c>0$ such that \eqref{20230228e1} holds.
\end{theorem}

\noindent Since $u_c$ is a local minimizer, it is natural to wonder:
\begin{itemize}
\item[({\bf Q1})] Whether $u_c$ is a normalized ground state solution? \\
Here normalized ground state means that $I(u_c)$ possesses the least energy among all the normalized solutions, namely
$$I(u_c)=\min\{I(u): u\in S_c,~ I\big|'_{S_c}(u)=0\}.$$
\item[({\bf Q2})] Whether there exists a second higher energy solution. This question can be regarded, in dimension two, as the Soave's open problem \cite[Remark 1.1]{Soave2020a} stated in higher dimensions.
\item[({\bf Q3})] Whether a normalized ground state solution is unique. Furthermore, what about the asymptotic behaviour of the normalized ground state as $c\rightarrow 0^+$, the so called zero mass case?
\item[({\bf Q4})] The asymptotic behaviour of higher energy solutions as the mass vanishes.
\end{itemize}

\medskip

\noindent We can answer the question ({\bf Q1}) in the following
\bt\lab{th:20230425-t1}
Assume in addition to the assumptions of Theorem \ref{theorem 1.1} that also $\mathbf{(f_3)}$ holds and let $c_0$ be sufficiently small. Then the normalized solution $u_c$ given in Theorem \ref{theorem 1.1} is a normalized ground state solution.
\et
\noindent Next we give a positive answer to question ({\bf Q2}).

\begin{theorem}\label{theorem 1.2}
Assume that $\mathbf{(f_1)}-\mathbf{(f_5)}$ hold. Define $S_{r,c}:=S_c\cap H_r^{1}(\R^2)$. Let $c\in (0,c_0)$ and $u_c$ be the normalized ground state solution given by Theorem \ref{th:20230425-t1}. Then Eq.\eqref{20230228e1}-\eqref{20230228e2} has a couple of solution $(\bar{u}_c,\bar{\lambda}_c)\in H^1(\R^2)\times\R$ such that $\bar{\lambda}_c>0$ and $\bar{u}_c\in S_{r,c}$ is a normalized mountain pass type critical point of $\left.I\right|_{S_{r,c}}$, which satisfies
$$
I(u_c)<0<I(\bar{u}_c)<I(u_c)+\frac{2\pi}{\alpha_0}.
$$
\end{theorem}

\noindent Before answering question ({\bf Q3}), let us replace condition $\mathbf{(f_1)}$ by the following:
\begin{itemize}
\item[$\mathbf{(f'_1)}$] $f\in C^1([0,+\infty)), f(s)>0$ for $s>0$ and $\lim\limits_{s\rightarrow 0^+}\frac{f'(s)}{s^{p-2}}:=\mu(p-1)>0$.
\end{itemize}

\bt\lab{th:20230426-xt1}
Under the assumptions $\mathbf{(f'_1)}$, $\mathbf{(f_2)}$ and $\mathbf{(f_3)}$, let $c_0$ be sufficiently small, then the following hold:
\begin{itemize}
\item[(i)] for a.e. $c\in (0,c_0)$, the normalized solution $u_c$ given in Theorem \ref{theorem 1.1} is the unique normalized ground state solution;
\item[(ii)]\begin{align*}
    \lambda_{c}^{\frac{1}{2-p}}u_{c}\left(\frac{x}{\sqrt{\lambda_{c}}}\right)\rightarrow U ~\hbox{in}~H^1(\R^2) ~\hbox{as}~c\rightarrow 0^+,
    \end{align*}
    where $U$ is the unique positive radial solution to
    \beq\lab{eq:20230524-e1}
    -\Delta U+U=\mu U^{p-1}~\hbox{in}~\R^2, \quad U(x)\rightarrow 0~\hbox{as}~|x|\rightarrow +\infty.
    \eeq
\end{itemize}
\et

\noindent Finally, we investigate the asymptotic behaviour of the mountain pass solutions $\bar{u}_c$ as $c\to0^+$. For technical reasons, we can not obtain a satisfactory result as in Theorem \ref{th:20230426-xt1}. However, we can give a partial answer to question ({\bf Q4}) in the sense of the following

\begin{theorem}\label{theorem 1.3}
Assume $\mathbf{(f_1)}-\mathbf{(f_3)}$ and let $c\in(0,c_0)$, $(\bar{u}_c,\bar{\lambda}_c)$ be the mountain pass type critical point of $\left.I\right|_{S_{r,c}}$ in Theorem \ref{theorem 1.2}, and $M_c$ the corresponding mountain pass level. Then, as $c\to0^+$ we have the following:
\begin{equation*}
\|\nabla\bar{u}_c\|_{2}^{2}\to\frac{4\pi}{\alpha_0}\ ,\quad
\int_{\R^2}f(\bar{u}_c)\bar{u}_c\mathrm{d}x\to\frac{4\pi}{\alpha_0}\ ,\quad
M_c\to\frac{2\pi}{\alpha_0} \ .
\end{equation*}
\end{theorem}

\noindent Let us contextualize our results within the literature, we begin with recalling the setting in the higher dimensional case. In presence of $L^2$-subcritical growth, $I$ turns out to be bounded from below on $S_c$ for all $c>0$, thus it is possible to find a global minimizer via minimization arguments, see \cite{Stuart1989,Stuart1982,Lions1984,Lions1984a,Shibata2014} and reference therein. For the $L^2$-supercritical case, $I$ is unbounded from below on $S_c$ for any $c>0$. In particular, it is worth to point out that, differently from the mass sub-critical case, the Gagligardo-Nirenberg inequality is not sufficient to guarantee the boundedness of Palais-Smale ($(PS)$ for short) sequences. Here, a major issue is to find a special bounded $(PS)$ sequence. This is done in \cite{Jeanjean1997} where a so-called Palais-Smale-Pohozaev ($(PSP)$ for short) sequence is constructed, which turns out to be bounded thanks to the property $\mathcal{P}(u_n)\to 0$, where
$$\mathcal{P}(u):=\|\nabla u\|_2^2+N\int_{\R^N}F(u)\mathrm{d}x -\frac{N}{2}\int_{\R^N}f(u)u\mathrm{d}x\ .$$
Under general assumptions, the functional recovers coercivity property once constrained to the Pohozaev manifold, see for instance \cite[Lemma 2.2]{Radulescu2024}. There are plenty of  contributions in the $L^2$-supercritical, let us just mention here \cite{Bartsch2013,Bartsch2017,Bieganowski2021,Ikoma2019,Jeanjean2020}. One also refers to the $L^2$-critical case when the boundedness from below does depend on the value $c>0$. We refer to \cite{Guo2014,Cheng2016} and reference therein.

\noindent More recently, a new approach to study normalized solution problem, which is called the global branch approach, has been developed by Bartsch-Zhong-Zou \cite{Bartsch2021} and Jeanjean-Zhang-Zhong \cite{Jeanjean2024}. Compared with the constrained variational method, this method does not depend on the geometry of the energy functional, and thus it enables one to handle in a unified way nonlinearities which are either mass sub-critical, mass critical or mass super-critical and can be also used to obtain existence of solution in non-variational problems \cite{Liu2024,Zeng2023}. This approach relies on the complete continuity of the corresponding operator, in order to apply the topological degree theory and related tools.

\noindent For the mass mixed case, the geometric structure of the energy functional becomes more complicated. In particular, in presence of Sobolev critical growth, the operator $T_\lambda:=(-\Delta+\lambda)^{-1}f$ is not completely continuous in $H_{rad}^{1}(\R^N)$, and thus the Leray-Shauder degree is not well-defined. So the global branch approach developed in \cite{Bartsch2021,Jeanjean2024} can not be directly applied. The investigation of normalized solutions to the Schr\"{o}dinger equations involving Sobolev critical exponents can be viewed as the counterpart of the Br\'ezis-Nirenberg problem within the framework of normalized solutions.

\noindent The authors in \cite{Soave2020a, Jeanjean2022,Wei2022,Chen2024,Radulescu2024}consider the case of $f(s)=\mu |s|^{q-2}s+|s|^{2^*-2}s$ with $\mu>0$ and obtain existence and asymptotic behaviour of solutions with respect to $\mu$.

\medskip

\noindent The aforementioned literature concerns the higher dimensional case $N\geq 3$. For the case $N=2$, where the Sobolev critical embedding is replaced by the Trudinger-Moser inequality, the plot thickens and some new extra difficulties enter the picture, see \cite{Cao1992,BezerradoO1997,Adachi2000,Ruf2005,Li2008,Adimurthi2007,Deng2021,Giacomoni2016,Lu2015,Araujo2022, Cassani2014, Cassani2021} and references therein.

\noindent Recently, normalized solutions in $\R^2$ with nonlinearities involving exponential critical growth have been considered by Alves, Ji and Miyagaki \cite{Alves2022} and Chang, Liu and Yan \cite{Chang2023}. Both \cite{Alves2022, Chang2023} assume the following condition:
\begin{itemize}
\item[$(f)$] There exist constants $p>4$ and $\mu>0$ such that
$$
sgn(t)f(t)\geq\mu|t|^{p-1},~~\forall~t\neq0\ .
$$
\end{itemize}
Condition $(f)$ where $\mu>0$ has to be sufficiently large is essential in \cite{Alves2022, Chang2023}, as it enables the authors to pull down the mountain pass level to recover the compactness. It is natural to ask whether this kind of restriction can be removed. Indeed, here by hypothesis $\mathbf{\mathbf{(f_4)}}$ which just demands for information of $f(s)$ when $s>0$ is large enough, will be our main ingredient to estimate the mountain pass level, by exploiting the Moser sequence, after a conformal change of variables. Assumption $\mathbf{\mathbf{(f_4)}}$ is the so-called \textit{FMR condition} from \cite{Figueiredo1995} and it plays a fundamental role to recover compactness as for the condition $I(u)<S^{N/2}/N$ in the higher dimensional Sobolev context.

\noindent All the above papers deal with the pure mass super-critical case and at the best of our knowledge the present paper is the first step in considering the mass-mixed case for general nonlinearities. When our article was almost finalised, we have been informed of the recent paper \cite{Chen2023} in which the special nonlinearity $f(u)=\mu |u|^{p-2}u+(e^{u^2}-1-u^2)u$ is considered. When $2<p<4$, this is also a mass mixed case. However, for this particular nonlinearity one may exploit, as in the higher dimensional double power case, the Pohozaev manifold $\mathcal{P}$ as a natural constraint. Indeed, let $T:H^{1}(\R^2)\times\R^+\to H^{1}(\R^2)$ with $T(u,s)=su(sx)$, which preserves the $L^2$-norm. Then $\Psi_{u}(s):=I(T(u,s))$ has exactly two critical points $0<s_1<s_2<\infty$. In particular, one has
    $$\frac{\ud}{\ud s}\Psi_{u}(s)=0\Rightarrow \frac{\ud^2}{\ud s^2}\Psi_{u}(s)\neq 0\ .$$
Thus, similar to the double power case, the corresponding Pohozaev manifold $\mathcal{P}$ is a natural constraint due to
    $$\mathcal{P}_0:=\left\{u\in S_c:\frac{\ud}{\ud s}\Psi_{u}(s)\Big|_{s=1}=0, \frac{\ud^2}{\ud s^2}\Psi_{u}(s)\Big|_{s=1}=0\right\}=\emptyset.$$
   In the present paper, the nonlinearities considered are more general and in particular we can not benefit of the Pohozaev manifold constraint, as it turns out in our case that $\mathcal{P}_0\neq\emptyset$, see Remark \ref{remark:20240705-r1}.

\subsection*{Overview} We first give a brief review on the Trudinger-Moser inequality in the plane together with some Br\'{e}zis-Lieb type results in the case of exponential critical growth which will be useful to derive compactness properties.
Since our nonlinearity $f$ has a mass subcritical behaviour at $0$, one expects to find a local minimizer for some suitable prescribed mass. We prove the local minima geometric structure in  Lemma \ref{20230301-lemma1}. In Section \ref{sec:20230802-4}, we prove that the normalized solution $u_c$ obtained in Theorem \ref{theorem 1.1} is indeed a normalized ground state solution,  we argue by contradiction arguments. In Section \ref{sec:20230802-5}, we prove that the normalized ground state solution $u_c$ converges to 0 in $H^1(\R^2)$, as the corresponding Lagrange multiplier $\lambda_c\to 0$. Then, the asymptotic behaviour of the normalized ground state solution is established by means of a blow up analysis. Finally, we prove that the normalized ground state solution is unique for almost every  prescribed small mass. In Section \ref{sec:20230307-1} and Section \ref{subsec:20230331}, we show that the functional  $I\big|_{S_c}$ (similarly for $I\big|_{S_{r,c}}$) possesses a mountain pass geometric structure at $u_c$ and we obtain an upper bound for the mountain pass level. Section \ref{subsec:20230307-1} is devoted to prove the $(PSP)_{M_c}$-condition for $I\big|_{S_{r,c}}$.Asymptotic analysis for the normalized mountain pass type solution as $c\rightarrow 0^+$, is carried out in Section \ref{sec:proofs-3}.

\subsection*{Notation}
\begin{itemize}
\item $L^s(\R^2)(1\leq s<\infty)$ denotes the Lebesgue space with the norm $\|u\|_s:=(\int_{\R^2}|u|^s\mathrm{d}x)^{\frac{1}{s}}$.
\item $H_{r}^{1}(\R^2)$ denotes the subspace of functions in $H^1(\R^2)$ which are radially symmetric with respect to 0, and $S_{r,c}:=S_c\cap H_r^{1}(\R^2)$.
\item The symbol $\|\cdot\|$ is used only for the norm in $H^1(\R^2)$.
\item Denoting by ${}^\ast$ the symmetric decreasing rearrangement of a $H^1(\R^2)$ function, we recall that, if $u\in H^1(\R^2)$, then $|u|\in H^1(\R^2),~|u|^*\in H_{r}^{1}(\R^2)$, with
$$
\|\nabla|u|^*\|_2\leq\|\nabla|u|\|_2\leq\|\nabla u\|_2.
$$
\item For any $x\in\R^2$ and $r>0,~B_r(x):=\{y\in\R^2:|y-x|<r\}$.
\item $\R^+:=(0,+\infty)$.
\item $C_i$ denote positive constants which may change from line to line.
\end{itemize}

\medskip
\section{Preliminaries}\lab{sec:20230331-2}
\subsection{A few remarks} Let us first make some comments about our assumptions on the nonlinearity $f$.
\br\lab{remark:20230425-br2}
Under the assumptions
$\mathbf{(f_1)}$-$\mathbf{(f_2)}$, let $r\in (2,p)$ and $q>4$ be fixed. Then one has the following: \begin{itemize}
\item[(i)] For any $\varepsilon>0$ and $\alpha>\alpha_0$, there exists $C_\varepsilon>0$ such that
\begin{align}\lab{eq:20230828-e1}
|f(s)|\leq\varepsilon|s|^{r-1}+C_\varepsilon|s|^{q-1}(e^{\alpha s^2}-1),~~\forall~s\in\R;
\end{align}
\item[(ii)]For any $\varepsilon>0$ and $\alpha>\alpha_0$, there exists $C_\varepsilon>0$ such that
\begin{align}\label{2.2}
|F(s)|\leq\varepsilon|s|^{r}+C_\varepsilon|s|^{q}(e^{\alpha s^2}-1),~~\forall~s\in\R;
\end{align}
\item[(iii)] It holds
\beq\lab{eq:20230425-be1}
\inf_{s>0}\frac{f(s)}{s^{p-1}}:=\varepsilon_0>0, \quad \inf_{s>0}\frac{F(s)}{s^{p}}:=\varepsilon_1>0;
\eeq
\item[(iv)] Suppose in addition $\mathbf{(f_3)}$ holds. Then, for any $\varepsilon>0$, there exists some $C_\varepsilon>0$ such that
\beq\lab{eq:20230425-we1}
F(s)\leq C_\varepsilon s^{p}+\varepsilon f(s)s,\quad \forall~ s\in \R.
\eeq
\end{itemize}
\er

\br\lab{remark:20230425-br1}
Under the assumptions $\mathbf{(f_1)}$-$\mathbf{(f_2)}$,
if $\lambda_c\in \R$, $u_c\in H^1(\R^2)$ is a nonnegative nontrivial solution to
\begin{align}\label{20230404e1}
-\Delta u+\lambda u=f(u)\quad\text{in}\quad\R^2.
\end{align}
Then $\lambda_c>0$ and $u_c$ is a positive radially decreasing function. Indeed, first by standard elliptic regularity, $u_c$ is a classical solution to \eqref{20230404e1}.
So if $\lambda_c\leq 0$, by \eqref{eq:20230425-be1}, if follows that
\begin{align*}
-\Delta u_c\geq \varepsilon_0 u_{c}^{p-1},~ u_c\geq 0\quad\text{in}\quad\R^2.
\end{align*}
Now, under the assumption $p>2$, we conclude that $u_c\equiv 0$ from \cite[Theorem 8.4]{Quittner2007}, and then a contradiction. Hence, $\lambda_c>0$. On the other hand, by the standard moving plain argument, one can see that $u_c$ is a positive radially decreasing function, see \cite[Theorem 2]{Gidas1981}.
\er

\noindent In view of Remark \ref{remark:20230425-br1}, we only need to work in the radial subspace
$$
H_r^{1}(\R^2):=\{u\in H^1(\R^2):u(x)=u(|x|)\}
$$
for which the embedding $H_{r}^{1}(\R^2)\hookrightarrow L^{t}(\R^2)$ is compact for all $t\in(2,\infty)$.

\br\lab{remark:20240705-r1}
Here we show that for the general mass mixed case we are considering, one has that the  corresponding Pohozaev manifold may degenerate at some point. We do this by constructing a counter example such that $\mathcal{P}_0\neq \emptyset$. Indeed, let
$\displaystyle f(s)=|s|s+|s|^{\frac{3}{2}}s+[2se^{s^2}-2s-2|s|^2 s]$.
A direct computation shows that
\begin{align*}
\frac{\mathrm{d}}{\mathrm{d}t}I(tu(tx))=&\|\nabla u\|_2^2 t-\frac{1}{3}\|u\|_3^3 -\frac{3}{7}\|u\|_{\frac{7}{2}}^{\frac{7}{2}} t^{\frac{1}{2}}
-\sum_{k=3}^{\infty}\frac{(2k-2)}{k!}\|u\|_{2k}^{2k}t^{2k-3}
\end{align*}
and
\begin{align*}
\frac{\mathrm{d}^2}{\mathrm{d}t^2}I(tu(tx))=&\|\nabla u\|_2^2  -\frac{3}{14}\|u\|_{\frac{7}{2}}^{\frac{7}{2}} t^{-\frac{1}{2}}
-\sum_{k=3}^{\infty}\frac{(2k-2)(2k-3)}{k!}\|u\|_{2k}^{2k}t^{2k-4}.
\end{align*}
We aim at finding some $u\in H^1(\R^2)$ such that
\begin{equation}\label{eq:20240525-e1}
\|\nabla u\|_2^2 -\frac{1}{3}\|u\|_3^3 -\frac{3}{7}\|u\|_{\frac{7}{2}}^{\frac{7}{2}}
-\sum_{k=3}^{\infty}\frac{(2k-2)}{k!}\|u\|_{2k}^{2k}=0
\end{equation}
and
\begin{equation}\label{eq:20240525-e2}
\|\nabla u\|_2^2  -\frac{3}{14}\|u\|_{\frac{7}{2}}^{\frac{7}{2}}
-\sum_{k=3}^{\infty}\frac{(2k-2)(2k-3)}{k!}\|u\|_{2k}^{2k}=0,
\end{equation}
which are equivalent to
\begin{equation}\label{eq:20240525-e3}
\begin{cases}
&\frac{1}{3}\|u\|_3^3 +\frac{3}{14}\|u\|_{\frac{7}{2}}^{\frac{7}{2}}-\sum_{k=3}^{\infty}\frac{(2k-2)^2}{k!}\|u\|_{2k}^{2k}=0,\\
&\|\nabla u\|_2^2 +\frac{1}{3}\|u\|_3^3
-\sum_{k=3}^{\infty}\frac{(2k-2)(4k-5)}{k!}\|u\|_{2k}^{2k}=0.
\end{cases}
\end{equation}
By $u=s(tv(tx)), v\in H^1(\R^2)\backslash\{0\}, t>0, s>0$, \eqref{eq:20240525-e3} yields
\begin{equation}\label{eq:20240525-e4}
\begin{cases}
\frac{1}{3}\|v\|_3^3 s^3t +\frac{3}{14}\|v\|_{\frac{7}{2}}^{\frac{7}{2}}s^{\frac{7}{2}}t^{\frac{3}{2}}
-\sum_{k=3}^{\infty}\frac{(2k-2)^2}{k!}\|v\|_{2k}^{2k}s^{2k}t^{2k-2}=0,\\
\|\nabla v\|_2^2s^2t^2  +\frac{1}{3}\|v\|_3^3 s^3t
-\sum_{k=3}^{\infty}\frac{(2k-2)(4k-5)}{k!}\|v\|_{2k}^{2k}s^{2k}t^{2k-2}=0,
\end{cases}
\end{equation}
that is
\begin{equation}\label{eq:20240525-e5}
\begin{cases}
\frac{1}{3}\|v\|_3^3  +\frac{3}{14}\|v\|_{\frac{7}{2}}^{\frac{7}{2}}s^{\frac{1}{2}}t^{\frac{1}{2}}
-\sum_{k=3}^{\infty}\frac{(2k-2)^2}{k!}\|v\|_{2k}^{2k}s^{2k-3}t^{2k-3}=0,\\
\|\nabla v\|_2^2\frac{t}{s}  +\frac{1}{3}\|v\|_3^3
-\sum_{k=3}^{\infty}\frac{(2k-2)(4k-5)}{k!}\|v\|_{2k}^{2k}s^{2k-3}t^{2k-3}=0.
\end{cases}
\end{equation}
Set $m=\frac{t}{s}, \ell=s^{\frac{1}{2}}t^{\frac{1}{2}}$, then the above equations become
\begin{equation}\label{eq:20240525-e6}
\begin{cases}
\frac{1}{3}\|v\|_3^3  +\frac{3}{14}\|v\|_{\frac{7}{2}}^{\frac{7}{2}}\ell
-\sum_{k=3}^{\infty}\frac{(2k-2)^2}{k!}\|v\|_{2k}^{2k}\ell^{4k-6}=0,\\
\|\nabla v\|_2^2m  +\frac{1}{3}\|v\|_3^3
-\sum_{k=3}^{\infty}\frac{(2k-2)(4k-5)}{k!}\|v\|_{2k}^{2k}\ell^{4k-6}=0.
\end{cases}
\end{equation}
Thus, $\ell>0$ solves
$$g_v(\ell):=\frac{1}{3}\|v\|_3^3  +\frac{3}{14}\|v\|_{\frac{7}{2}}^{\frac{7}{2}}\ell
-\sum_{k=3}^{\infty}\frac{(2k-2)^2}{k!}\|v\|_{2k}^{2k}\ell^{4k-6}=0.$$
By $\displaystyle \lim_{\ell\rightarrow 0^+}g_v(\ell)=\frac{1}{3}\|v\|_3^3$ and $\displaystyle \lim_{\ell\rightarrow +\infty}g_v(\ell)=-\infty<0$, we conclude that there exists some $\ell=\ell(v)>0$ such that $g_v(\ell)=0$.
Then we obtain
$$m=m(v):=\frac{\sum_{k=3}^{\infty}\frac{(2k-2)(4k-5)}{k!}\|v\|_{2k}^{2k}\ell^{4k-6}-\frac{1}{3}\|v\|_3^3}{\|\nabla v\|_2^2}.$$
We \textit{claim} that $m>0$.
 In fact, if not, by $m\leq 0$, we have
\begin{align*}
&\sum_{k=3}^{\infty}\frac{(2k-2)(4k-5)}{k!}\|v\|_{2k}^{2k}\ell^{4k-6}\leq \frac{1}{3}\|v\|_3^3\\
=&\sum_{k=3}^{\infty}\frac{(2k-2)^2}{k!}\|v\|_{2k}^{2k}\ell^{4k-6}
-\frac{3}{14}\|v\|_{\frac{7}{2}}^{\frac{7}{2}}\ell\\
<&\sum_{k=3}^{\infty}\frac{(2k-2)^2}{k!}\|v\|_{2k}^{2k}\ell^{4k-6},
\end{align*}
thus
$\sum_{k=3}^{\infty}\frac{(2k-2)(2k-3)}{k!}\|v\|_{2k}^{2k}\ell^{4k-6}< 0$, which is a contradiction since $\frac{(2k-2)(2k-3)}{k!}>0$ for any $k\geq 3$.

\noindent Finally, for any $v\in H^1(\R^2)\backslash\{0\}$, let $\ell(v)$ and $m(v)$ be as above, and define
$$t=t(v):=\sqrt{m(v)}\ell(v), s=s(v):=\frac{1}{\sqrt{m(v)}}\ell(v)\ .$$
Then, $u:=s(tv(tx))$ solves \eqref{eq:20240525-e3}, which implies that $u\in \mathcal{P}_0$.
\er

\subsection{Functional setting}
\noindent Let  $H^1_0(\Omega)$ be the classical Sobolev space, completion of smooth compactly supported functions with respect to the Dirichlet norm $\|\nabla \cdot\|_2$, when $\Omega$ is a bounded subset of $\R^N$, and with respect to the complete Sobolev  norm $(\|\nabla \cdot\|^2_2+\|\cdot \|_2^2)^{1/2}$, when the domain is unbounded and in particular for $\Omega =\R^N$.

\noindent If $N\geq 3$,  the  classical Sobolev embedding theorem reads as follows
\begin{equation}\label{Si}
H^1_0(\Omega)\hookrightarrow L^{2^*}(\Omega) \ , \ \text{ namely }\ \|u\|_{2^*}\leq \frac{1}{S}\, \|\nabla u\|_2\ ,
\end{equation}
where $2^*:=\frac{2N}{N-2}$ is the critical Sobolev exponent and the constant $S$ in \eqref{Si} is sharp \cite{Talenti1976}.

\noindent When $N=2$ is the so-called Sobolev limiting case. One has the embedding $H^1_0(\Omega)\hookrightarrow L^p(\Omega)$ for all $1\leq p < \infty$, though $H^1_0(\Omega) \not \subset L^{\infty}(\Omega)$.
The maximal degree of summability for functions in $H^1_0(\Omega)$ was established independently by Poho\v{z}aev \cite{Pohozaev1965} and Trudinger \cite{Trudinger1967} and is of exponential type, in a suitable Orlic\v{z} class of functions, namely
\begin{equation}\label{orlicz}
u \in H^1_0(\Omega) \ \Longrightarrow \ \displaystyle\int_\Omega (e^{\alpha|u|^2}-1)\mathrm{~d} x<\infty
, \quad \forall \alpha>0\ .
\end{equation}
Starting from the seminal work of Moser \cite{Moser1970} in which a sharp version of \eqref{orlicz} is established, the Pohozaev-Trudinger embedding has been further developed during the last fifty years, in particular the first extension of \eqref{orlicz} to unbounded domains is due to Cao \cite{Cao1992} for functions with bounded Sobolev's norm in the following form
\begin{equation}\label{caoineq}
\sup_{\|\nabla u\|_2\leq 1,\, \|u\|_2\leq M}\int_{\R^2}\left(e^{\alpha u^2}-1\right)\mathrm{~d} x\leq C(\alpha)\|u\|_2<\infty\,\text{ if } \,\alpha<4\pi.
\end{equation}
Thereafter, several sharp versions have been proved and extensions in many directions. In particular, the borderline case in which $\alpha=4\pi$ remained uncovered until Ruf in \cite{Ruf2005} established the following inequality which is sharp in the sense of Moser \cite{Moser1970} (so-called Trudinger-Moser type inequalities):
\begin{equation}\label{Ri}
\sup_{\|\nabla u\|_2^2+\:\|u\|_2^2\leq 1}\int_{\mathbb R^2}\left(e^{\alpha
	u^2}-1\right)\mathrm{~d} x\leq \widetilde{C}(\alpha)<\infty  \iff  \alpha\leq 4\pi\ .
\end{equation}
Generalizations to the limiting  case for the Sobolev embedding of $W^{1, N}(\mathbb{R}^N)$ have been established in \cite{Adachi2000, BezerradoO1997, Li2008} as follows:
\begin{align}\label{20240621e1}
\sup _{\|u\|_{W^{1, N}(\mathbb{R}^N)} \leq 1} \int_{\mathbb{R}^N} \phi_N(\alpha|u|^{\frac{N}{N-1}}) \mathrm{~d} x \begin{cases}\leq C(\alpha, N)^{\frac{N}{N-1}}, & \text { if } \alpha \leq \alpha_N, \\ =\infty, & \text { if } \alpha>\alpha_N,\end{cases}
\end{align}
where $\alpha_N=N\omega_{N-1}^{\frac{1}{N-1}}$,~$\omega_{N-1}$ is the measure of the unit sphere in $\R^N$ and
$$
\phi_N(t)=e^t-\sum_{j=0}^{N-2} \frac{t^j}{j !}=\sum_{j=N-1}^{\infty} \frac{t^j}{j !},~t \geq 0,
$$
see also \cite{Cassani2014}.

\noindent Let us next prove some useful results which will be helpful in the sequel.

\begin{lemma}\label{lemma 3.7}
Let $\bar{u}_n\rightharpoonup \bar{u}_c$ weakly in $H^{1}(\R^2)$ and $\limsup\limits_{n\to\infty}\|\nabla(\bar{u}_n-\bar{u}_c)\|_{2}^{2}<\frac{4\pi}{\alpha_0}$. Then for $\alpha>\alpha_0$ close to $\alpha_0$, we have that $\big\{e^{\alpha \bar{u}_n^2}-1\big\}$ is bounded in $L^t(\R^2)$ provided $t>1$ is close to 1.
\end{lemma}
\begin{proof}
Let us prove first that for any $\sigma>0$, there exists some $C=C_\sigma$, which only depends on $\sigma$, such that
\beq\lab{eq:20240714-e1}
|a+b|^{2k}\leq (1+\sigma)^k |b|^{2k}+C^k |a|^{2k}, \forall a,b\in \R, \forall k\geq 1.
\eeq
Indeed, \eqref{eq:20240714-e1} is equivalent to
\beq\lab{eq:20240714-e2}
1\leq (1+\sigma)^k (1-s)^{2k}+C^k s^{2k}=:h(s), \forall s\in [0,1].
\eeq 
By a direct computation, one can see that $h(s)$ attains its minimum at some $s_0\in (0,1)$, which is given by
\beq\lab{eq:20240714-e3}
(1+\sigma)^k (1-s_0)^{2k-1}=C^ks_{0}^{2k-1}.
\eeq
So, it has a minimum
\beq\lab{eq:20240714-e4}
h(s_0)=C^k s_{0}^{2k-1}.
\eeq
Hence, if we take $C\geq s_{0}^{-\frac{2k-1}{k}}$, \eqref{eq:20240714-e2} is true.
By inspection we have that 
$$\lim_{k\rightarrow +\infty}s_{0}^{-\frac{2k-1}{k}}=\frac{1+\sigma}{(\sqrt{1+\sigma}-1)^2},$$
so we can take $C=C_\sigma$ sufficiently large and independent of $k\geq 1$.

\noindent Now, for $\alpha>\alpha_0$ and $t>1$ close to 1, we may assume that
\beq\lab{eq:20240714-e5}
\limsup\limits_{n\to\infty}t\alpha\|\nabla(\bar{u}_n-\bar{u}_c)\|_{2}^{2}<4\pi.
\eeq
Set $\bar{v}_n:=\bar{u}_n-\bar{u}_c$. Let $\sigma>0$ be sufficiently small and let $C=C_\sigma$ be given by \eqref{eq:20240714-e1}.
Then,
\begin{align*}
&\int_{\R^2}|e^{\alpha \bar{u}_n^2}-1|^t \mathrm{d}x
\leq\int_{\R^2} (e^{t\alpha \bar{u}_n^2}-1) \mathrm{d}x
=\int_{\R^2} (e^{t\alpha |\bar{v}_n+\bar{u}_c|^2}-1) \mathrm{d}x\\
=&\int_{\R^2}\sum_{k=1}^{\infty}\frac{(t\alpha)^k}{k!} |\bar{v}_n+\bar{u}_c|^{2k} \mathrm{d}x\\
\leq&\int_{\R^2}\sum_{k=1}^{\infty}\frac{(t\alpha)^k}{k!}
\left[(1+\sigma)^k |\bar{v}_n|^{2k}+C^k |\bar{u}_c|^{2k}\right]\mathrm{d}x\\
=&\int_{\R^2}\sum_{k=1}^{\infty}\frac{(t\alpha)^k}{k!}(1+\sigma)^k |\bar{v}_n|^{2k}\mathrm{d}x
+\int_{\R^2}\sum_{k=1}^{\infty}\frac{(t\alpha)^k}{k!}C^k |\bar{u}_c|^{2k}\mathrm{d}x\\
=&\int_{\R^2}(e^{t\alpha(1+\sigma)\bar{v}_n^2}-1) \mathrm{d}x
+\int_{\R^2}(e^{C\bar{u}_c^2}-1) \mathrm{d}x.
\end{align*}
By taking $\sigma>0$ small enough such that 
\beq\lab{eq:20240714-e6}
\limsup\limits_{n\to\infty}t\alpha(1+\sigma)\|\nabla(\bar{u}_n-\bar{u}_c)\|_{2}^{2}<4\pi,
\eeq
\eqref{orlicz} and \eqref{caoineq} yield the boundedness of $\{\int_{\R^2}|e^{\alpha \bar{u}_n^2}-1|^t \mathrm{d}x\}$.
\end{proof}

\bl\label{corollary 3.1}
Let $\{\bar{u}_n\}\subset H^{1}(\R^2)$ be such that $\bar{u}_n\rightharpoonup\bar{u}_c$ weakly in $H^{1}(\R^2)$ and $\limsup\limits_{n\to\infty}\|\nabla(\bar{u}_n-\bar{u}_c)\|_{2}^{2}<\frac{4\pi}{\alpha_0}$. Under the assumptions $\mathbf{(f_1)}$ and $\mathbf{(f_2)}$, the following holds
\begin{align*}
\int_{\R^2}F(\bar{u}_n)\mathrm{d} x=\int_{\R^2}F(\bar{u}_c)\mathrm{d} x+\int_{\R^2}F(\bar{u}_n-\bar{u}_c)\mathrm{d}x+o_n(1).
\end{align*}
Suppose in addition that $f\in C^1(\R)$, and that for any $\varepsilon>0$, there exists $C_\varepsilon>0$ such that
\begin{align*}
|f^{\prime}(s)|\leq\varepsilon|s|^{\bar{p}-2}+C_\varepsilon|s|^{q-2}(e^{\alpha s^2}-1),\quad\forall~s\in\R,
\end{align*}
where $\alpha>\alpha_0$ and $q>4$. Then we have
\begin{align*}
\int_{\R^2}f(\bar{u}_n)\bar{u}_n\mathrm{d} x=\int_{\R^2}f(\bar{u}_c)\bar{u}_c\mathrm{d} x+\int_{\R^2}f(\bar{u}_n-\bar{u}_c)(\bar{u}_n-\bar{u}_c)\mathrm{d} x+o_n(1).
\end{align*}
\el
\begin{proof}
Extracting a subsequence if necessary, we may assume $\bar{u}_n\rightarrow \bar{u}_c$ a.e. in $\R^2$. 
Let $\theta_n(x)\in [0,1]$ be such that
\begin{align*}
\int_{\R^2}\left|F(\bar{u}_n)-F(\bar{u}_n-\bar{u}_c)\right|\mathrm{d} x=\int_{\R^2}|f(\bar{u}_n-(1-\theta_n)\bar{u}_c)\bar{u}_c|\mathrm{d} x.
\end{align*}
Recalling \eqref{eq:20230828-e1},
\begin{align*}
&|f(\bar{u}_n-(1-\theta_n)\bar{u}_c)\bar{u}_c|\\
\leq&\varepsilon|\bar{u}_n-(1-\theta_n)\bar{u}_c|^{r-1}|\bar{u}_c|+C_\varepsilon|\bar{u}_n-(1-\theta_n)\bar{u}_c|^{q-1}[e^{\alpha|\bar{u}_n-(1-\theta_n)\bar{u}_c|^{2}}-1]|\bar{u}_c|.
\end{align*}
By Lemma \ref{lemma 3.7}, $\{e^{\alpha \bar{u}_n^2}-1\}_{n=1}^{\infty}$ is bounded in $L^t(\R^2)$ provided $\alpha>\alpha_0$ is close to $\alpha_0$ and $t>1$ is close to one. Now, proceeding as in Lemma \ref{lemma 3.7}, one can also prove that  $\{e^{\alpha|\bar{u}_n-(1-\theta_n)\bar{u}_c|^2}-1\}_{n=1}^{\infty}$ is bounded in $L^{t_1}(\R^2)$ with $1<t_1<t.$

\noindent We \textit{claim} that $\{f(\bar{u}_n-(1-\theta_n)\bar{u}_c)\}$ stays bounded in $L^\tau(\R^2)$ for any $1<\tau<t_1.$ In order to prove this claim, we only need to show that $$\big\{|\bar{u}_n-(1-\theta_n)\bar{u}_c|^{q-1}(e^{\alpha|\bar{u}_n-(1-\theta_n)\bar{u}_c|^2}-1\big\}_{n=1}^{\infty}$$ stays bounded in $L^\tau(\R^2)$ for any $1<\tau<t_1.$ By the Sobolev embedding $H^1(\R^2)\hookrightarrow L^p(\R^2)$ for any $2\leq p<+\infty$ and the H\"{o}lder inequality, one has
\begin{align*}
&\int_{\R^2}|\bar{u}_n-(1-\theta_n)\bar{u}_c|^{(q-1)\tau}\big[e^{\alpha|\bar{u}_n-(1-\theta_n)\bar{u}_c|^2}-1\big]^\tau\mathrm{d}x\\
&\leq\big(\int_{\R^2}|\bar{u}_n-(1-\theta_n)\bar{u}_c|^{(q-1)\tau\frac{t_1}{t_1-\tau}}\mathrm{d} x\big)^{\frac{t_1-\tau}{t_1}}\\
&~~~~\big(\int_{\R^2}[e^{\alpha|\bar{u}_n-(1-\theta_n)\bar{u}_c|^2}-1]^{t_1}\mathrm{d} x\big)^{\frac{\tau}{t_1}}\\
&\leq C\|e^{\alpha|\bar{u}_n-(1-\theta_n)\bar{u}_c|^2}-1\|_{t_1}^{\tau}<\infty,
\end{align*}
and the claim is proved.

\noindent Noting that $\bar{u}_c\in L^{\tau^{\prime}}(\R^2)$ with $\tau^{\prime}=\frac{\tau}{\tau-1}>2,$ it follows that $\{f(\bar{u}_n-(1-\theta_n)\bar{u}_c)\bar{u}_c\}$ is bounded in $L^1(\R^2).$ Define $B_{R}^{c}:=\{x\in\R^2:|x|>R\}$, we have
\begin{align*}
\int_{B_{R}^{c}}|f(\bar{u}_n-(1-\theta_n)\bar{u}_c)\bar{u}_c|\mathrm{d}x\leq\|f(\bar{u}_n-(1-\theta_n)\bar{u}_c)\|_\tau\|\bar{u}_c\|_{L^{\tau^{\prime}}(B_{R}^{c})}\to0,
\end{align*}
uniformly in $n$, as $R\to+\infty.$ Furthermore, for any $\Lambda\subset\R^2,$ we also have
\begin{align*}
\int_{\Lambda}|f(\bar{u}_n-(1-\theta_n)\bar{u}_c)\bar{u}_c|\mathrm{d}x\leq\|f(\bar{u}_n-(1-\theta_n)\bar{u}_c)\|_{L^\tau(\Lambda)}\|\bar{u}_c\|_{L^{\tau^{\prime}}(\Lambda)}\to0,
\end{align*}
uniformly in $n$, as $\text{meas}(\Lambda)\to 0.$ So $\{F(\bar{u}_n)-F(\bar{u}_n-\bar{u}_c)\}$ possesses the uniform integrability condition, and then by Vitali's convergence theorem and the fact $F(0)=0,$ we obtain
\begin{align*}
\lim\limits_{n\to\infty}\int_{\R^2}[F(\bar{u}_n)-F(\bar{u}_n-\bar{u}_c)]\mathrm{d}x=\int_{\R^2}\lim\limits_{n\to\infty}[F(\bar{u}_n)-F(\bar{u}_n-\bar{u}_c)]\mathrm{d}x=\int_{\R^2}F(\bar{u}_c)\mathrm{d}x.
\end{align*}

\noindent Similarly, if $f\in C^1$ satisfies 
\begin{align*}
|f^{\prime}(s)|\leq\varepsilon|s|^{\bar{p}-2}+C_\varepsilon|s|^{q-2}(e^{\alpha s^2}-1),\quad\forall~s\in\R,
\end{align*}
with $\alpha>\alpha_0, q>4$, one can show also in this case that
\begin{align*}
\int_{\R^2}f(\bar{u}_n)\bar{u}_n\mathrm{d} x=\int_{\R^2}f(\bar{u}_c)\bar{u}_c\mathrm{d} x+\int_{\R^2}f(\bar{u}_n-\bar{u}_c)(\bar{u}_n-\bar{u}_c)\mathrm{d} x+o_n(1).
\end{align*}
\end{proof}

\noindent The previous Br\'{e}zis-Lieb type Lemma can be improved if the sequences are radially symmetric.
\begin{corollary}\label{corollary 3.2}
Let $\{\bar{u}_n\}\subset H_{r}^{1}(\R^2)$ be such that $\bar{u}_n\rightharpoonup \bar{u}_c$ weakly in $H_{r}^{1}(\R^2)$ and $\limsup\limits_{n\to\infty}\|\nabla(\bar{u}_n-\bar{u}_c)\|_{2}^{2}<\frac{4\pi}{\alpha_0}$. Under the assumptions $\mathbf{(f_1)}$ and $\mathbf{(f_2)}$, we have
\begin{align*}
\int_{\R^2}f(\bar{u}_n)\bar{u}_n\mathrm{d}x=\int_{\R^2}f(\bar{u}_c)\bar{u}_c\mathrm{d}x+o_n(1),
\end{align*}
and
\begin{align*}
\int_{\R^2}F(\bar{u}_n)\mathrm{d}x=\int_{\R^2}F(\bar{u}_c)\mathrm{d}x+o_{n}(1).
\end{align*}
\end{corollary}

\begin{proof}
Recalling \eqref{eq:20230828-e1}, by Lemma \ref{lemma 3.7}, one can show that $\{f(\bar{u}_n)\}$ is bounded in $L^\tau(\R^2)$  for $1<\tau<\min\{t,2\}$ close to one. By the radial compact embedding, we have that $\bar{u}_n\to \bar{u}_c$ in $L^{\tau^{\prime}}$ with $\tau^{\prime}=\frac{\tau}{\tau-1}>2.$ Thus,
\begin{align*}
\int_{\R^2}|f(\bar{u}_n)(\bar{u}_n-\bar{u}_c)|\mathrm{d}x\leq\|f(\bar{u}_n)\|_\tau\|\bar{u}_n-\bar{u}_c\|_{\tau^{\prime}}=o_n(1).
\end{align*}
On the other hand, up to subsequence, $f(\bar{u}_n)\rightharpoonup f(\bar{u}_c)$ weakly in $L^{\tau}(\R^2),$ which implies
\begin{align*}
\int_{\R^2}[f(\bar{u}_n)-f(\bar{u}_c)]\bar{u}_c\mathrm{d}x=o_n(1).
\end{align*}
Hence,
$$
\big|\int_{\R^2}f(\bar{u}_n)\bar{u}_n-f(\bar{u}_c)\bar{u}_c\mathrm{d}x\big|\leq\big|\int_{\R^2}[f(\bar{u}_n)-f(\bar{u}_c)]\bar{u}_c\mathrm{d}x\big|+\big|\int_{\R^2}f(\bar{u}_n)(\bar{u}_n-\bar{u}_c)\mathrm{d}x\big|=o_n(1)
$$
which implies that
\begin{align*}
\int_{\R^2}f(\bar{u}_n)\bar{u}_n\mathrm{d}x=\int_{\R^2}f(\bar{u}_c)\bar{u}_c\mathrm{d}x+o_n(1),
\end{align*}

\noindent Thanks to Corollary \ref{corollary 3.1}, in order to prove $\displaystyle \int_{\R^2}F(\bar{u}_n)\mathrm{d}x=\int_{\R^2}F(\bar{u}_c)\mathrm{d}x+o_{n}(1).
$, it is sufficient to show that
$
\displaystyle\int_{\R^2}F(\bar{u}_n-\bar{u}_c)\mathrm{d}x=o_n(1).
$
According to \eqref{2.2}, we only need to prove that
\begin{align}\label{3.29}
\int_{\R^2}|\bar{u}_n-\bar{u}_c|^q[e^{\alpha|\bar{u}_n-\bar{u}_c|^2}-1]\mathrm{d}x=o_n(1).
\end{align}
Indeed, by Lemma \ref{lemma 3.7} and the H\"{o}lder inequality, there exists $C>0$ independent of $n$ such that
\begin{align*}
&\int_{\R^2}|\bar{u}_n-\bar{u}_c|^q[e^{\alpha|\bar{u}_n-\bar{u}_c|^2}-1]\mathrm{d}x\notag\\
&\leq\||\bar{u}_n-\bar{u}_c|^q\|_{t^{\prime}}\|e^{\alpha|\bar{u}_n-\bar{u}_c|^2}-1\|_t\notag\\
&\leq C\||\bar{u}_n-\bar{u}_c|^q\|_{t^{\prime}},
\end{align*}
where $t^{\prime}=\frac{t}{t-1}.$ By the radial compact embedding again, we get that $\bar{u}_n\to \bar{u}_c$ in $L^{qt^{\prime}}(\R^2)$, and then \eqref{3.29} holds.
\end{proof}

\section{Normalized local minimizer}\lab{sec3}
\noindent In the study of normalized solutions to Schr\"{o}dinger type equations in the whole space, a crucial point is the behaviour of the related energy functional with respect to the $L^2$-norm invariance group action. Let us then introduce the scaling operator
$T:H^{1}(\R^2)\times\R^+\to H^{1}(\R^2)$ defined as
$$T(u,s)=su(sx)
$$
 which preserves the $L^2$-norm, namely $\|T(u,s)\|_2=\|u\|_2,~\forall~s>0$. Let us then consider the fiber
\begin{align}\label{20230308e1}
\Psi_{u}(s)=I(T(u,s))=\frac{s^2}{2}\|\nabla u\|_{2}^{2}-\frac{1}{s^2}\int_{\R^2}F(su)\mathrm{d}x
\end{align}
and note that if $(u,\lambda)$ is any solution to \eqref{20230228e1}-\eqref{20230228e2}, then $\frac{d}{ds}\Psi_u(s)|_{s=1}=0$, that is,
\begin{equation}\label{Pmanifold}
\mathcal{P}(u)=\|\nabla u\|_2^2+2\int_{\R^2}F(u)\mathrm{d}x-\int_{\R^2}f(u)u\mathrm{d}x=0.
\end{equation}
The set
$$
\{u: \mathcal P(u)=0\}
$$
is the so-called {\it{Pohozeav manifold}}, and will play a key role in our discussion.
Since $f$ has exponential critical growth at infinity, one has $\Psi_{u}(s)\to -\infty$  as $s\to +\infty$, so that $I$ is unbounded from below on the  constraint $S_c$. On the other hand, the mass sub-critical behaviour of  $f(s)$ near $0$ suggests the presence of a local minimum, at least for some suitable given mass.
\subsection{Local minimum structure for small mass}
\begin{lemma}\label{20230301-lemma1}
Assume that $\mathbf{(f_1)}$ and $\mathbf{(f_2)}$ hold.
Then there exist $c_0>0$, $\rho^*\in (0,\frac{4\pi}{\alpha_0})$ and $\tau_0>0$ such that for any $c\in (0,c_0)$,
\begin{align}\label{20230228e3}
m_c:=\inf _{u \in V_c} I(u)<0<\tau_0<\inf _{u \in \partial V_c} I(u),
\end{align}
where $V_c:=\big\{u \in S_{c}:\|\nabla u\|_{2}^2<\rho^*\big\}$.
\end{lemma}
\begin{proof}
Let $0<\rho^*<\frac{4\pi}{\alpha_0}$ and $M>0$ be sufficiently large.
By Cao's inequality \eqref{caoineq}, for any $t>1$ close enough to $1$ and for any $\alpha>\alpha_0$ close to $\alpha_0$, we have
\begin{equation*}
\sup_{\|\nabla u\|_2^2\leq \rho^*, \|u\|_2^2\leq M} \|e^{\alpha u^2}-1\|_t <\infty.
\end{equation*}
Let $q>4$ and by H\"{o}lder's inequality we have
\begin{align}\label{3.6}
\int_{\R^2}|u|^{q}\big(e^{\alpha u^2}-1\big)\mathrm{d}x\leq C\|u\|_{qt^{\prime}}^{q}, \quad \forall~u\in H^1(\R^2),~\|\nabla u\|_2^2\leq \rho^*,~\|u\|_2^2\leq M,
\end{align}
 where $t^{\prime}:=\frac{t}{t-1}$ is the Young conjugate exponent of $t$. \\
By  \eqref{2.2} and \eqref{3.6}, for any $\varepsilon>0$ and $r\in(2,p)$, there exists $C_\varepsilon>0$ such that
\beq\lab{eq:20221225-e1}
\int_{\R^2}F(u)\mathrm{d}x\leq \varepsilon\|u\|_{r}^{r}+C_\varepsilon\|u\|_{qt^{\prime}}^{q},~\forall~u\in H^1(\R^2),~\|\nabla u\|_2^2\leq \rho^*,~\|u\|_2^2\leq M.
\eeq
We next apply the following Gagliardo-Nirenberg inequality,
\begin{align*}
\|u\|_s\leq C_{s}\|\nabla u\|_{2}^{\gamma_s}\|u\|_{2}^{1-\gamma_s},~\gamma_s:=1-\frac{2}{s},~u\in H^1(\R^2), \ \ \forall~s\in(2, +\infty).
\end{align*}
Then, for any $u \in H^1(\R^2)$ with $\|u\|_2^2=c^2\leq M$ and $\|\nabla u\|_2^2=\rho\leq \rho^*$, from \eqref{eq:20221225-e1} we obtain
\begin{align}\label{20230302e1}
I(u)&=\frac{1}{2}\|\nabla u\|_{2}^2-\int_{\R^2}F(u)\mathrm{d}x\nonumber\\
&\geq\frac{1}{2}\|\nabla u\|_{2}^2-\varepsilon\|u\|_{r}^{r}-C_\varepsilon\|u\|_{qt^{\prime}}^{q}\\
&\geq\frac{1}{2}\|\nabla u\|_{2}^2-\varepsilon C_{r}^{r}\|\nabla u\|_{2}^{r-2}\|u\|_{2}^{2}-C_\varepsilon C_{qt'}^{q}\|\nabla u\|_{2}^{q-\frac{2}{t'}}\|u\|_{2}^{\frac{2}{t'}}\nonumber\\
&=\|\nabla u\|_{2}^2\Big[\frac{1}{2}-\varepsilon c^2 C_{r}^{r}\|\nabla u\|_{2}^{\bar{p}-4}-C_\varepsilon c^{\frac{2}{t'}} C_{qt'}^{q}\|\nabla u\|_{2}^{q-4+\frac{2}{t}}\Big]\nonumber\\
&:=\rho f(c,\rho)\nonumber,
\end{align}
where
\begin{equation}\lab{eq:20230425-xe2}
f(c,\rho):=\frac{1}{2}-c^2a\rho^{\frac{r-4}{2}}- c^{\frac{2}{t'}} b\rho^{\frac{q}{2}-2+\frac{1}{t}},\quad 0<\rho\leq \rho^*,
\end{equation}
and where $a, b$ are positive constants which depend on $\varepsilon, r, q, t$ but not on $c$. Let us choose $t$  sufficiently close to $1$, such that $\frac{q}{2}-2+\frac{1}{t}>1$. For any $c>0$ fixed, consider the map  $\rho\mapsto g_c(\rho):=f(c,\rho)$ for all $ \rho>0$.  Since $r<4$, it is easy to verify that $g_c(\rho)$ has a unique global maximum point, $\rho_c$,  on $(0,+\infty)$,  given by
\begin{align}\lab{eq:20230425-xe1}
\rho_c=Ac^\tau,
\end{align}
where $A:=A(\varepsilon,r,q,t) >0$ and $\tau:=\tau(r,q,t)=\frac 4{t(q-r)+2}>0$, and $\rho_c\leq \rho^*$ for all $c\in (0,(\rho^*/A)^{\frac 1\tau})$.
By inspection,
\begin{align*}
\max\limits_{\rho>0}g_c(\rho)=g_c(\rho_c)=\frac{1}{2}-(aA^{\frac{r-4}{2}}+bA^{\frac{q}{2}-2+\frac{1}{t}})c^{2\frac{r-2+t(q-r)}{2+t(q-r)}}>0,
\end{align*}
if $c\in (0, c_1]$ where $c_1$ is small enough. Recalling \eqref{eq:20230425-xe2}, it is easy to see that there exists some $c_2>0$ such that
\begin{align}\label{20230723e1}
f(c,\rho^*)>0, \forall~ c\in (0,c_2].
\end{align}
Now take $c_0=\min\{c_1,c_2,\sqrt{M}\}$: since, for any fixed $c$, $g_c(\rho)$ decreases in $[\rho_c, \rho^*]$ we have $g_c(\rho)\geq f(c,\rho^*)>0$ for all $c\in (0,c_0]$ and $\rho\in [\rho_c,\rho^*]$.
Furthermore, by \eqref{20230302e1} and \eqref{20230723e1}, we obtain
\beq\lab{eq:20230828-xbe1}
\inf_{u \in \partial V_c}I(u)\geq\rho^* f(c,\rho^*)>\tau_0:=\rho^* f(c_0,\rho^*)>0,~\forall~c\in (0,c_0).
\eeq
Let us now prove that $\inf\limits_{u \in V_c} I(u)<0$. Consider $u\in S_c$ be fixed: recalling \eqref{eq:20230425-be1} and \eqref{20230308e1}, we have
 \begin{align*}
 I(T(u,s))=&\frac{s^2}{2}\|\nabla u\|_2^2-\frac{1}{s^2}\int_{\R^2}F(su)\mathrm{d}x\\
 \leq&\frac{s^2}{2}\|\nabla u\|_2^2 - \frac{\varepsilon_1}{s^2}\|su\|_{p}^{p}\\
 =&s^{p-2}[\frac{1}{2}s^{4-p}\|\nabla u\|_2^2-\varepsilon_1\|u\|_{p}^{p}].
 \end{align*}
As $s>0$ is small enough, we have $\|\nabla T(u,s)\|_2^2=s^2\|\nabla u\|_2^2<\rho^*$ and $I(T(u,s))<0$ due to $p<4$.
Hence, for $s>0$ sufficiently small, $T(u,s)\in V_c$ and thus $m_c:=\inf\limits_{u\in V_c}I(u)<0$.
\end{proof}

\br\lab{remark:20230425-r1}
By the previous proof, it is clear that
\begin{align*}
I(u)>0, \forall~ u\in S_c, c\in (0,c_0)~\hbox{and}~\|\nabla u\|_2^2\in [\rho_c,\rho^*],
\end{align*}
where $\rho_c$ is given by \eqref{eq:20230425-xe1}. Hence, if $m_c$ is attained by some $u_c$, then $\|\nabla u_c\|_2^2<\rho_c$. Thus by $\rho_c\rightarrow 0$ as $c\rightarrow 0^+$, we conclude that  $u_c\rightarrow 0$ in $H^1(\R^2)$ as $c\rightarrow 0^+$.
\er
%
\subsection{Existence of a local minimizer}
Let $\rho^*$ be given as in Lemma \ref{20230301-lemma1} and define
\begin{align*}
\tilde{V}_c:=\{u\in H^1(\R^2): 0<\|u\|_2\leq c, \|\nabla u\|_2^2<\rho^*\}
\end{align*}
together with
\begin{align*}
\tilde{m}_c:=\inf_{u\in \tilde{V}_c}I(u).
\end{align*}

\bl\lab{lemma:20230426-l1}
Under the assumptions $\mathbf{(f_1)}$-$\mathbf{(f_2)}$, there exists $c_0>0$ small enough such that for any $c\in (0,c_0)$ one has:
\begin{itemize}
\item[(i)]$\tilde{m}_c$ is attained by some positive radial function $u\in S_{r,c}$;
\item[(ii)]$m_c=\tilde{m}_c$ and it is strictly decreasing with respect   $c\in(0,c_0)$.
\end{itemize}
\el
\begin{proof}  By the very definition $\tilde{m}_c\leq m_c<0$ for $c\in (0,c_0)$.
On the other hand, for any $u\in \tilde{V}_c$, let $u^*$ be its spherical decreasing
rearrangement and since $f(s)=0$ for $s\leq 0$, we get
$$\int_{\R^2}F(u)\mathrm{d}x\leq \int_{\R^2}F(|u|)\mathrm{d}x=\int_{\R^2}F(u^*)\mathrm{d}x~~\hbox{and}~~\|\nabla u^*\|_2^2\leq \|\nabla u\|_2^2.$$
Hence, $I(u^*)\leq I(u)$. So, we can consider a radial decreasing minimizing sequence $\{u_n\}\subset \tilde{V}_c\cap H_r^1(\R^2)$, with $I(u_n)\rightarrow \tilde{m_c}$. Since it is bounded in $H^1(\R^2)$, up to a subsequence we may assume $u_n\rightharpoonup u$ in $H_r^1(\R^2)$.
Furthermore, applying \eqref{2.2} one has
\begin{align*}
\int_{\R^2}F(u_n)\mathrm{d}x\rightarrow \int_{\R^2}F(u)\mathrm{d}x,
\end{align*}
by applying Cao's inequality \eqref{caoineq} and thanks to the choice of $\rho^*$, besides the embedding  $H_{r}^{1}(\R^2)\hookrightarrow L^{t}(\R^2)$ is compact for all $t\in(2,\infty)$. By Remark \ref{remark:20230425-r1}, we may assume that $\|\nabla u_n\|_2^2< \rho_c<\rho^*$, so that also $\|\nabla u\|_2^2< \rho^*$. Furthermore, by
\begin{align*}
I(u)\leq \liminf_{n\rightarrow \infty}I(u_n)=\tilde{m}_c\leq m_c<0,
\end{align*}
we conclude that  $u\neq 0$ and thus $u\in \tilde{V}_c$. Then, by definition,  $I(u)\geq \tilde{m}_c$: this combined with the previous inequality yields $\tilde{m}_c$ is attained by $u$.

\noindent Next, we will prove that $u\in S_{r,c}$, which implies directly $m_c=\tilde{m}_c$, since $m_c\leq I(u)=\tilde{m}_c\leq m_c$. \newline Suppose by contradiction that $\|u\|_2<c$. Then,
\begin{align*}
\tilde{m}_c=I(u)=\min_{\|\nabla v\|_2^2<\rho^*, 0<\|v\|_2<c} I(v).
\end{align*}
Since the set $\{\|\nabla v\|_2^2<\rho^*, 0<\|v\|_2<c\}$ is open, we conclude that $u>0$ satisfies
\begin{align*}
-\Delta u=f(u)~\hbox{in}~\R^2,
\end{align*}
which is a contradiction, see Remark \ref{remark:20230425-br1}.

\noindent Let us now prove that $m_c$ decreases strictly in $(0,c_0)$. By definition, $m_c$ is not increasing in $(0,c_0)$; if by contradiction it is not strictly decreasing, then we could find $0<c_1<c_2<c_0$ such that $m_c\equiv m_{c_1}, \forall~ c\in [c_1,c_2]$. Let $u_1\in S_{r,c_1}$ be such that attains $m_{c_1}$: then $\|u_1\|_2\leq c_1<c_2$ and it attains also $m_{c_2}$. So $-\Delta u_1=f(u_1)$ in $\R^2$, which is again  contradiction.
\ep

\vspace*{0,3cm}

\noindent We are now ready to prove Theorem \ref{theorem 1.1}, which establishes the existence of normalized local minimizers.

\vspace*{0,3cm}

\noindent
{\bf Proof of Theorem \ref{theorem 1.1}:}
By Lemma \ref{20230301-lemma1}, under the assumptions $\mathbf{(f_1)}$ and $\mathbf{(f_2)}$, there exists $c_0>0$ small enough such that $I$ possesses a local minimum geometric structure:
$$\inf _{u \in V_c} I(u)<0<\inf _{u \in \partial V_c} I(u),$$
where $V_c:=\big\{u \in S_{c}:\|\nabla u\|_{2}^2<\rho^*\big\}$ for some $\rho^*\in (0, \frac{4\pi}{\alpha_0})$.
In particular, by Lemma \ref{lemma:20230426-l1}, $m_c:=\inf\limits_{u \in V_c} I(u)$ is attained by some positive radial function, denoted by $u_c$. Then there exists some $\lambda_c\in \R$ such that $(u_c,\lambda_c)$ is a normalized solution to Eq.\eqref{20230228e1}. By the rearrangement inequality again, $I(u_c^*)\leq I(u_c)$. Hence, $u_c=u_c^*$ which decreases with respect to $|x|$. In particular, by Remark \ref{remark:20230425-br1}, we have that $\lambda_c>0$. Finally, the strong maximum principle implies that $u_c$ is positive.
\hfill$\Box$

\subsection{Normalized ground state solution}\lab{sec:20230802-4}
Let us prove now that, assuming  in addition the condition $\mathbf{(f_3)}$, actually the normalized solution $u_c$ found in Theorem \ref{theorem 1.1} is indeed a normalized ground state solution.

\begin{lemma}\label{lemma 20230813}
 Under the assumptions $\mathbf{(f_1)}, \mathbf{(f_2)}$ and $\mathbf{(f_3)}$ the normalized solution $u_c$ found in Theorem \ref{theorem 1.1} is a normalized ground state solution for any $c\in (0, c_0)$, where $c_0$ is sufficiently small.
\end{lemma}
\begin{proof}
Let us argue by contradiction. Suppose that there exists $c_n\downarrow 0$ and a corresponding  sequence $\{v_n\}$ such that
\beq\lab{eq:20230426-ze1}
\|v_n\|_2^2=c_n^2,~I\big|'_{S_{c_n}}(v_n)=0,~ I(v_n)< m_{c_n}.
\eeq
By Lemma \ref{20230301-lemma1}, without loss of generality, we may assume that $m_{c_n}<0$ for all $n$.
By
\begin{align*}
I(v_n)=\frac{1}{2}\|\nabla v_n\|_2^2-\int_{\R^2}F(v_n)\mathrm{d}x<m_{c_n}<0
\end{align*}
and
\begin{align*}
\mathcal{P}(v_n)=\|\nabla v_n\|_2^2+2\int_{\R^2}F(v_n)\mathrm{d}x -\int_{\R^2}f(v_n)v_n\mathrm{d}x=0,
\end{align*}
recalling \eqref{eq:20230425-we1}, we have that
\begin{align*}
\frac{1}{2}\|\nabla v_n\|_2^2<\int_{\R^2}F(v_n)\mathrm{d}x \leq \varepsilon \int_{\R^2}f(v_n)v_n\mathrm{d}x +C_\varepsilon\|v_n\|_{p}^{p},
\end{align*}
where $2<p<4$. Then
\begin{align*}
\|\nabla v_n\|_2^2+2\int_{\R^2}F(v_n)\mathrm{d}x<&4\left(\varepsilon \int_{\R^2}f(v_n)v_n \mathrm{d}x +C_\varepsilon \|v_n\|_{p}^{p}\right)\\
=&4\varepsilon \left(\|\nabla v_n\|_2^2+2\int_{\R^2}F(v_n)\mathrm{d}x\right)+4C_\varepsilon \|v_n\|_{p}^{p}.
\end{align*}
Combine this with the Gagliardo-Nirenberg inequality to get
\begin{align*}
(1-4\varepsilon) \|\nabla v_n\|_2^2 \leq &(1-4\varepsilon)\left(\|\nabla v_n\|_2^2+2\int_{\R^2}F(v_n)\mathrm{d}x\right)\\
\leq&4C_\varepsilon \|v_n\|_{p}^{p}\\
\leq& 4C_{\varepsilon, p} \|\nabla v_n\|_{2}^{p-2} \|v_n\|_2^2.
\end{align*}
Choose $\varepsilon \in (0, \frac{1}{4})$: noting that $\|v_n\|_2^2=c_n^2\rightarrow 0$, we conclude that
$$\|\nabla v_n\|_2^2\leq o(1)\|\nabla v_n\|_{2}^{p-2}.$$
Since $2<p<4$, we have that $\|\nabla v_n\|_2\rightarrow 0$, as $n\rightarrow \infty$. Hence, for $n$ large, $\|\nabla v_n\|_2^2<\rho^*$, and thus
 $v_n\in V_{c_n}$, which implies that $I(v_n)\geq \inf\limits_{u\in V_{c_n}}I(u)=m_{c_n}$, which contradicts \eqref{eq:20230426-ze1}.
\end{proof}

\subsection{Asymptotic behavior of the local minimizer as $c\to 0^+$}\lab{sec:20230802-5} In this subsection we answer to question ({\bf{Q3}}). In order to investigate the asymptotic behavior of $u_c$ as $c\to 0^+$, we first give an estimate of the Lagrange multiplier $\lambda_c$ for $c>0$ small, which will be useful in the proof of Theorem \ref{th:20230426-xt1}.
\bl\lab{lemma:20230426-zl1}
Under the assumptions $\mathbf{(f_1)}$-$\mathbf{(f_3)}$,
for $c\in (0,c_0)$, let $(u_c,\lambda_c)$ be the normalized ground state solution given by Theorem  \ref{th:20230425-t1}. Then $\lambda_c\rightarrow 0$, as $c\rightarrow 0^+$.
\el
\begin{proof}
Arguing as in the proof of Theorem \ref{th:20230425-t1},  $\|\nabla u_c\|_2\rightarrow 0$ as $c\rightarrow 0^+$, see also  Remark \ref{remark:20230425-r1}.
Furthermore, following the same lines as above,
\begin{align*}
\|\nabla u_c\|_2^2+\lambda_c c^2 =&\int_{\R^2}f(u_c)u_c \mathrm{d}x\\
=&\|\nabla u_c\|_2^2+2\int_{\R^2}F(u_c)\mathrm{d}x\\
\leq& 4 \left(\varepsilon \int_{\R^2}f(u_c)u_c \mathrm{d}x +C_\varepsilon \|u_c\|_{p}^{p}\right)\\
=&4\varepsilon \left(\|\nabla u_c\|_2^2+\lambda_c c^2\right)+4 C_\varepsilon\|u_c\|_{p}^{p},
\end{align*}
where $p\in (2,4)$. So, by the Gagliardo-Nirenberg inequality,
\begin{align*}
(1-4\varepsilon)\left(\|\nabla u_c\|_2^2+ \lambda_c c^2\right) \leq 4C_\varepsilon \|u_c\|_{p}^{p} \leq C_{\varepsilon,p} \|\nabla u_c\|_{2}^{p-2} c^2.
\end{align*}
Choose $\varepsilon<\frac{1}{4}$: we obtain
\begin{align*}
0<\lambda_c<\frac{(1-4\varepsilon)\left(\|\nabla u_c\|_2^2+ \lambda_c c^2\right)}{(1-4\varepsilon)c^2}\leq \frac{C_{\varepsilon,p} }{1-4\varepsilon} \|\nabla u_c\|_{2}^{p-2}\rightarrow 0,~\hbox{as}~c\rightarrow 0^+,
\end{align*}
since $p>2$ and $\|\nabla u_c\|_2\rightarrow 0$. We remark that the formula above also implies $\frac{\|\nabla u_c\|_2^2}{c^2}\rightarrow 0$ as $c\rightarrow 0^+$.
\ep

\br\lab{remark:20230814-r2}
The  proofs of the asymptotic results \cite[Theorem 4.5, Theorem 4.6,Theorem 5.1 and Lemma 6.4]{Jeanjean2024} strongly rely on the property $\limsup\limits_{c\rightarrow 0^+}\|u_c\|_\infty<+\infty$. However, in the planar framework, differently from the higher dimensional case, there is no corresponding limit equation at infinity, which prevents us to apply the same method as in \cite[Lemma 3.25 and Lemma 4.12]{Radulescu2024}.
\er

\bl\lab{lemma:20230814-bl1}
Under the assumptions $\mathbf{(f_1)}$-$\mathbf{(f_3)}$,
for $c\in (0,c_0)$, let $(u_c,\lambda_c)$ be the normalized ground state solution given by Theorem  \ref{th:20230425-t1}, then we have $\|u_c\|_\infty\rightarrow 0$, as $c\rightarrow 0^+$.
\el
\begin{proof} Let $k>1$ be fixed: we claim that $\limsup\limits_{c\rightarrow 0^+}\int_{\R^2}(e^{ku_c^2}-1)\mathrm{d}x\leq M$ and $M>0$ is independent of $k$.  Indeed, for any $\alpha\in (0,4\pi)$ and $k>1$ fixed, one can see that $k\|u_c\|_2^2<\alpha$ provided $c>0$ small enough since $u_c\rightarrow 0$ in $L^2(\R^2)$. On the other hand, by the proof of Lemma \ref{lemma:20230426-zl1}, we have that $\frac{\|\nabla u_c\|_2^2}{\|u_c\|_2^2}=\frac{\|\nabla u_c\|_2^2}{c^2}\rightarrow 0$ as $c\rightarrow 0^+$.
Hence, by Cao's inequality \eqref{caoineq},
\beq\lab{eq:20230814-xe1}
\limsup_{c\rightarrow 0^+}\int_{\R^2}(e^{ku_c^2}-1)\mathrm{d}x\leq \limsup_{c\rightarrow 0^+}\int_{\R^2} \Big(e^{\alpha \big(\frac{u_c}{\|u_c\|_2}\big)^2}-1\Big)\mathrm{d}x\leq C(\alpha).
\eeq

Under the assumptions $\mathbf{(f_1)}$ and $\mathbf{(f_2)}$,  let $\tilde{p}\in (2,p)$ and $q>4$ be fixed. Then for any $r>1$ and $\varepsilon>0$, one can find some $C_\varepsilon>0$ such that
$$
|f(s)|^r\leq \varepsilon |s|^{r(\tilde{p}-1)}+C_\varepsilon|s|^{r(q-1)}(e^{\alpha r s^2}-1),~~\forall~s\in\R.
$$
Noting that $u_c\rightarrow 0$ in $L^t(\R^2)$ for any $t\in [2,+\infty)$, combining with \eqref{eq:20230814-xe1} and H\"older inequality, one can deduce that
$$
\limsup_{c\rightarrow 0^+} \|f(u_c)\|_{r}=0.
$$
Recalling that $\lambda_c\rightarrow 0$, as $c\rightarrow 0^+$ and by $u_c\rightarrow 0$ in $L^t(\R^2)$ for any $t\in [2,+\infty)$, we conclude that for any $r>1$,
$$
\limsup_{c\rightarrow 0^+} \|\Delta u_c\|_r=0.
$$
Notice that $u_c$ is also a strong solution of problem \eqref{20230228e1} by standard bootstrap arguments. Furthermore, if $r>1$ the $L^r$-norm of the Laplacian, is equivalent to the full seminorm $\sum_{i,j}\|\partial_{i,j}u\|_2^2$, so that
$$
\|u_c\|_{W^{2,r}(\R^2)}\leq C(r) \left(\|\Delta u_c\|_r+\|u_c\|_{H^1(\R^2)}\right)\rightarrow 0, ~\hbox{as}~c\rightarrow 0^+.
$$
Since $W^{2,r}(\R^2)\hookrightarrow L^\infty(\R^2)$, we conclude that
$$
\|u_c\|_\infty \leq C(r)\|u_c\|_{W^{2,r}(\R^2)}\rightarrow 0, ~\hbox{as}~c\rightarrow 0^+.
$$
\ep

\br\lab{remark:20230426-xr1}
Under our assumptions $\mathbf{(f_1)}$-$\mathbf{(f_3)}$, by Lemma \ref{lemma:20230426-zl1}, $\lambda_c\rightarrow 0$ as $c\to 0^+$, and by Lemma \ref{lemma:20230814-bl1}, $\|u_c\|_\infty=u_c(0)\rightarrow 0$ as $c\rightarrow 0^+$. Applying the same arguments as in  \cite[Lemma 4.2, Lemma 4.3]{Jeanjean2024}, up to  replacing $\mathbf{(f_1)}$ by $\mathbf{(f'_1)}$, one can  also prove that there exist $C_1,C_2>0$ such that $C_1\|u_c\|_{\infty}^{p-2}\leq \lambda_c \leq C_2\|u_c\|_{\infty}^{p-2}$, see \cite[Lemmas 4.2 and Lemma 4.3]{Jeanjean2024}.
\er

\br\lab{remark:20240628-br1}
Provided $\lambda_0$ is sufficiently small, one has the following facts:
\begin{itemize}
\item[(1)] The fixed frequency problem \eqref{20230228e1} has a unique positive solution, which is denoted by $u_\lambda$, $\forall \lambda\in (0,\lambda_0)$, see \cite[Theorem 5.1]{Jeanjean2024};
\item[(2)] $\{(\lambda, u_\lambda): \lambda\in (0,\lambda_0)\}$ is an analytic curve in $\R\times H^1(\R^N)$. Indeed, by \cite[Theorem 4.1]{Jeanjean2024}, $v_\lambda(x):=\lambda^{\frac{1}{2-p}}u_\lambda(\frac{x}{\sqrt{\lambda}})\rightarrow U$ in $H^1(\R^2)$ as $\lambda\rightarrow 0^+$, where $U$ is the unique positive radial solution to equation \eqref{eq:20230524-e1}.
    Thus, by the non-degeneracy of $U$, we obtain the non-degeneracy of $u_\lambda$ for $\lambda>0$ small enough. Then by the implicit function theorem, we can deduce the analytical property;
\item[(3)] By (2) above, $u_\lambda$ is differentiable respect to $\lambda\in (0, \lambda_0)$;
\item[(4)] Define $c_\lambda:=\|u_\lambda\|_2$, then $c_\lambda$ is differentiable with respect to $\lambda\in (0, \lambda_0)$;
\item[(5)]By \cite[Theorem 5.1 and Proposition 6.1]{Jeanjean2024}, $u_\lambda$ is the least action solution, which is also the mountain pass solution;
\item[(6)] Define $I_\lambda(u):=I(u)+\frac{1}{2}\lambda \|u\|_2^2$, then $M(\lambda):=I_\lambda(u_\lambda)$ is the mountain pass level for the fixed frequency problem \eqref{20230228e1} provided $\lambda\in (0,\lambda_0)$;
\item[(7)] $M(\lambda)$ increases strictly with respect to $\lambda\in \R^+$, and it is differentiable at $(0,\lambda_0)$;
\item[(8)] There exists $\gamma>2$ such that $f(u_\lambda(x))u_\lambda(x)>\gamma F(u_\lambda(x)), \forall (\lambda,x)\in (0,\lambda_0)\times \R^N$, see the proof of \cite[Lemma 6.2]{Jeanjean2024}.
\end{itemize}
\er

\bl\lab{lemma:20240628-l1}
 There exists $\lambda_1>0$ such that $c_\lambda$ increases with respect to $\lambda\in (0,\lambda_1)$.
\el
\begin{proof}
Now, for any given $\lambda\in (0,\lambda_0)$, there exists some $\eta_\lambda>0$ small enough such that $(\lambda-\eta_\lambda, \lambda+\eta_\lambda)\subset (0,\lambda_0)$. Then, for any $\tau\in(\lambda-\eta_\lambda, \lambda+\eta_\lambda)$, one can show that there exists an unique $s=s(\tau)>0$ such that
$su_\lambda\in \mathcal{N}_{\tau}$, the so-called Nehari manifold corresponding to the fixed frequency problem \eqref{20230228e1} with $\lambda=\tau$, defined by
$$\mathcal{N}_{\tau}:=\{u\in H^1(\R^2)\backslash\{0\}: \|\nabla u\|_2^2+\tau \|u\|_2^2-\int_{\R^2}f(u)u\mathrm{d}x=0\}.$$
Precisely, $s=s(\tau)$ is determined by
$$\|\nabla u_\lambda\|_2^2 s+\tau \|u_\lambda\|_2^2 s-\int_{\R^2}f(su_\lambda)u_\lambda \mathrm{d}x=0,$$
which implies that $s=s(\tau)$ is differentiable with respect to $\tau\in (\lambda-\eta_\lambda, \lambda+\eta_\lambda)$.

\noindent By $su_\lambda\in \mathcal{N}_{\tau}$,
\beq\lab{eq:20240617-e2}
M(\tau)=\inf_{u\in \mathcal{N}_\tau}I_\tau(u)\leq I_{\tau}(su_\lambda)=:h(\tau).
\eeq
Moreover, for any $\tau\in (\lambda-\eta_\lambda, \lambda)$,
$$\frac{\mathrm{d}}{\mathrm{d}t}I_{\tau}(tu_\lambda)\Big|_{t=1}
=\frac{\mathrm{d}}{\mathrm{d}t}I_{\lambda}(tu_\lambda)\Big|_{t=1}-\frac{\mathrm{d}}{\mathrm{d}t}\left\{\frac{1}{2}(\lambda-\tau)\|u_\lambda\|_2^2 t^2\right\}\Big|_{t=1}
=-(\lambda-\tau)\|u_\lambda\|_2^2<0.$$
Hence $s(\tau)<1$ and $s\rightarrow 1$, as $\tau\uparrow \lambda$.

\noindent By a direct computation, denote $s'(\tau)$ by $s'$, we have that
\begin{align*}
&\frac{\mathrm{d}}{\mathrm{d}\tau}h(\tau)
=\frac{\mathrm{d}}{\mathrm{d}\tau}I_{\tau}(su_\lambda)
=\frac{\mathrm{d}}{\mathrm{d}\tau}\left\{\frac{1}{2}[\|\nabla u_\lambda\|_2^2 +\tau \|u_\lambda\|_2^2]s^2-\int_{\R^2}F(su_\lambda) \mathrm{d}x\right\}\\
=&[\|\nabla u_\lambda\|_2^2 +\tau \|u_\lambda\|_2^2] s s' +\frac{1}{2}\|u_\lambda\|_2^2 s^2-\int_{\R^2}f(su_\lambda)u_\lambda s' \mathrm{d}x.
\end{align*}
So,
\begin{align*}
&\frac{\mathrm{d}}{\mathrm{d}\tau}h(\tau)\Big|_{\tau=\lambda}
=\frac{1}{2}\|u_\lambda\|_2^2+[\|\nabla u_\lambda\|_2^2 +\lambda \|u_\lambda\|_2^2-\int_{\R^2}f(u_\lambda)u_\lambda \mathrm{d}x]s'
=\frac{1}{2}\|u_\lambda\|_2^2>0.
\end{align*}
Hence, for $\eta_\lambda>0$ sufficiently small, combining the above formula with \eqref{eq:20240617-e2},we have that
$$M(\tau)\leq h(\tau)<h(\lambda)=M(\lambda), \forall~ \tau\in (\lambda-\eta_\lambda, \lambda).$$
Thus
\begin{align*}
M'(\lambda)=&\lim_{\tau\uparrow \lambda}\frac{M(\lambda)-M(\tau)}{\lambda-\tau}
\geq\lim_{\tau\uparrow \lambda}\frac{h(\lambda)-h(\tau)}{\lambda-\tau}
=h'(\lambda)=\frac{1}{2}\|u_\lambda\|_2^2.
\end{align*}
Now, we argue by contradiction.
Suppose that there is no $\lambda_1>0$ such that $c_\lambda$ increases with respect to $\lambda\in(0,\lambda_1)$. Then there exists sequences $c_n\downarrow 0$ and $0<\lambda_{1,n}<\lambda_{2,n}\downarrow 0$ such that
    $I(u_{\lambda_{1,n}})=I(u_{\lambda_{2,n}})=m_{c_n}$
and
\beq\lab{eq:20240615bu-e2}
\|u_\lambda\|_2>c_n, \forall~ \lambda\in (\lambda_{1,n},\lambda_{2,n}).
\eeq
However, by
$$
M(\lambda_{2,n})-M(\lambda_{1,n})=m_{c_n}+\frac{1}{2}\lambda_{2,n}c_n^2-m_{c_n}-\frac{1}{2}\lambda_{1,n}c_n^2=\frac{c_n^2}{2}(\lambda_{2,n}-\lambda_{1,n}),
$$
we can find some $\lambda_n\in (\lambda_{1,n},\lambda_{2,n})$ such that
$$\frac{1}{2}\|u_{\lambda_n}\|_2^2\leq M'(\lambda_n)=\frac{M(\lambda_{2,n})-M(\lambda_{1,n})}{(\lambda_{2,n}-\lambda_{1,n})}
=\frac{1}{2}c_n^2.$$
That is, $\|u_{\lambda_n}\|_2\leq c_n$ which contradicts \eqref{eq:20240615bu-e2}.
\end{proof}

\br\lab{remark:20240628-r1}
Let $c_0>0$ be sufficiently small and define
\beq\lab{eq:20240628-e1}
\underline{\lambda}_c:=\min\{\lambda\in (0,\lambda_1):c_\lambda=c\},~\overline{\lambda}_c:=\max\{\lambda\in (0,\lambda_1):c_\lambda=c\}.
\eeq
We stress that $\underline{\lambda}_c$ and $\overline{\lambda}_c$ are well defined due to Lemma \ref{lemma:20240628-l1}. Set
\beq\lab{eq:20240628-e2}
\lambda_c:=\frac{\underline{\lambda}_c+\overline{\lambda}_c}{2},
\eeq
then $\lambda_c$ is well defined for $c\in (0,c_0)$.
In particular, let
$$\Lambda_c:=\{\lambda\in (0,\lambda_1): \|u_\lambda\|_2=c\},$$
then again by Lemma \ref{lemma:20240628-l1}, one can see that $\Lambda_c:=[\underline{\lambda}_c, \overline{\lambda}_c]$.
We remark that $\lambda_c$ increases strictly with respect to $c\in (0,c_0)$. Thus, $\lambda_c$ is differentiable with respect to almost every $c\in (0,c_0)$, which implies that
$$\Lambda_c=\{\lambda_c\}=\{\underline{\lambda}_c\}=\{\overline{\lambda}_c\}$$
is a single point set for a.e. $c\in (0,c_0)$.
\er

\vskip 0.2in
\noindent
{\bf Proof of Theorem \ref{th:20230426-xt1}.} By Remark \ref{remark:20240628-r1}, for almost every $c\in (0,c_0)$, $\Lambda_c=\{\lambda_c\}$ is a single point set. On the other hand, by Lemma \ref{lemma:20230426-zl1}, $\lambda_c\in (0,\lambda_0)$ if $c_0$ is chosen suitably small. Then, combining this with \cite[Lemma 6.2]{Jeanjean2024}, we obtain that for almost every $c\in (0,c_0)$, $u_c$ is the unique normalized ground state.
The conclusion (i) is proved.

\noindent The conclusion (ii) follows from \cite[Theorem 4.1]{Jeanjean2024}, see also Remark \ref{remark:20240628-br1}-(2) above.
\hfill$\Box$

\br\lab{remark:20240628-br2}
Recalling Lemma \ref{lemma:20230426-l1}-(ii), $m_c$ decreases strictly with respect to $c\in (0,c_0)$. For the homogeneous case $f(s)=\mu s^{p-1},s>0$,  by a direct computation, one can deduce that
$\lambda_c=\frac{4m_c}{(p-4)c^2}$, which is uniquely determined by $c$. So, it will be an interesting question whether the conclusion (i) in Theorem \ref{th:20230426-xt1} can be strengthen to all $c\in (0,c_0)$ for the general nonlinearity $f$.
\er

\section{Normalized mountain pass solution}\lab{sec:proofs-2}
\noindent The goal of this Section is to prove the existence of a second couple $(\bar u_c, \bar \lambda_c)$ which is a solution to problem \eqref{20230228e1}-\eqref{20230228e2} of mountain pass type.  \\
The local minimizer $u_c \in S_{r,c}$ obtained in Theorem \ref{theorem 1.1} can be characterized by
$$
I(u_c)=\inf _{u \in V_c} I(u)=m_c<0.
$$

\noindent Let
\begin{align}\label{20230307e2}
M_c:=\inf _{\gamma \in \Gamma_c} \max _{t \in[0, \infty)} I(\gamma(t)),
\end{align}
where
\begin{align*}
\Gamma_c:=\big\{\gamma \in C([0, \infty), S_{r,c}): \gamma(0)=u_c,~\exists~t_\gamma>0 \text { s.t. } I(\gamma(t))<2m_c, \forall~ t \geq t_\gamma\big\}.
\end{align*}

Following the idea introduced by Jeanjean in  \cite{Jeanjean1997} we consider the functional $\tilde{I}: H^{1}(\R^2)\times\R^+\to\R$ defined by
\beq\lab{eq:20230828-xe1}
\tilde{I}(u,s):=I(T(u,s))=\Psi_{u}(s)=\frac{s^2}{2}\|\nabla u\|_{2}^{2}-\frac{1}{s^2}\int_{\R^2}F(su)\mathrm{d}x.
\eeq
For any fixed $c>0$, we also define
$$
\widetilde{M}_c:=\inf _{\tilde{\gamma} \in \tilde{\Gamma}_c}\max\limits_{t \in[0, \infty)}\tilde{I}(\tilde{\gamma}(t)),
$$
where
$$
\tilde{\Gamma}_c:=\{\tilde{\gamma} \in C([0, \infty), S_{r,c}\times\mathbb{R}^{+}): \tilde{\gamma}(0)=(u_c,1),~\exists~t_{\tilde{\gamma}}>0 \text { s.t. } \tilde{I}(\tilde{\gamma}(t)) <2m_c,~\forall~t \geq t_{\tilde{\gamma}}\}.
$$

\subsection{Mountain pass geometry}\lab{sec:20230307-1}

Recalling the definition of $M_c$ given in \eqref{20230307e2}, we have
\begin{lemma}\label{20230307-lemma-3}
For any $c \in(0, c_0)$,~$\tilde{I}$ has a mountain pass geometry at the level $\widetilde{M}_c$ and $M_c=\widetilde{M}_c\geq \tau_0$, where $\tau_0>0$ is given by Lemma \ref{20230301-lemma1}.
\end{lemma}
\begin{proof}
Let $\gamma \in \Gamma_c$. Since $\tilde{\gamma}(t)=(\gamma(t),1) \in \tilde{\Gamma}_c$ and $\tilde{I}(\tilde{\gamma}(t))=I(\gamma(t))$ for all $t \in [0,+\infty)$, we have that $M_c \geq \widetilde{M}_c$.
 Conversely, for any $\tilde{\gamma}(t)=(v(t),s(t)) \in \tilde{\Gamma}_c$, we have that $v(0)=u_c,~s(0)=1$ and there exists a $t_{\tilde{\gamma}}>0$ such that $\tilde{I}(\tilde{\gamma}(t))<2m_c$ for all $t\geq t_{\tilde{\gamma}}$. Set $\gamma(t)=T(v(t),s(t))$. Then, $\gamma$ is continuous from $[0,+\infty)$ into $S_{r,c}$ and
$$
\gamma(0)=T(v(0),s(0))=v(0)=u_c,
$$
$$
I(\gamma(t))=I(T(v(t),s(t)))=\tilde{I}(v(t),s(t))=\tilde{I}(\tilde{\gamma}(t)) <2m_c,\quad\forall~t \geq t_{\tilde{\gamma}}.
$$
Hence, $\gamma \in \Gamma_c$ and $\tilde{I}(\tilde{\gamma}(t))=I(T(v(t),s(t)))=I(\gamma(t))$. Thus, $\widetilde{M}_c \geq M_c$. Hence, we finally conclude that $M_c=\widetilde{M}_c$.

\noindent Next, let us prove that $M_c\geq \tau_0$ for all $c \in(0, c_0)$.
Indeed, let $\gamma \in \Gamma_c$ be arbitrary and $\gamma(0)=u_c \in V_c$. Now, for $t>0$ sufficiently large, since $I(\gamma(t))<2 m_c$, necessarily in view of \eqref{20230228e3}, $\gamma(t)\notin V_c$. By continuity of $\gamma$ there exists a $t_0>0$ such that $\gamma(t_0) \in \partial V_c$ and using again \eqref{20230228e3} we conclude.

\noindent At this point observing that
$$
\max\{\tilde{I}(\tilde{\gamma}(0)), \tilde{I}(\tilde{\gamma}(t_\gamma))\}=\max\{I(\gamma(0)), I(\gamma(t_\gamma))\}<0,
$$
it follows that $\tilde{I}$ has a mountain pass geometry at level $\widetilde{M}_c$ for all $0<c<c_0$.
\end{proof}

\medskip
\noindent We adopt the following definition from  \cite[Definition 2.21]{Willem1996}.
\begin{definition}\lab{def:20230215-d1}
Let $X$ be a Banach space and $\Psi\in C^1(X,\R)$,~$d\in \R$,$~\mathcal{P}\in C(X,\R)$. A sequence $\{u_n\}\subset X$ is called a $(PSP)_d$ sequence for $\Psi$ if
$$
\Psi(u_n)\rightarrow d,~\Psi'(u_n)\rightarrow 0,~\mathcal{P}(u_n)\to 0~\hbox{as}~n\rightarrow \infty.
$$
The functional $\Psi$ satisfies the $(PSP)_d$ condition if any $(PSP)_d$ sequence $\{u_n\}\subset X$ has a convergent subsequence.
\end{definition}

\noindent In dimension two, applying arguments similar to  \cite[Proposition 2.2 and Lemma 2.4]{Jeanjean1997}, one can  prove, starting from Lemma \ref{20230307-lemma-3}, the following result
\begin{lemma}\label{lemma 3.3}
There exists $\{\bar{u}_n\}\subset S_{r,c}$ such that
$$
I(\bar{u}_n)\to M_c,\quad \left.I\right|_{S_{r,c}}^{\prime}(\bar{u}_n)\to 0,\quad\mathcal{P}(\bar{u}_n)\to 0~~\hbox{as}~~n\to \infty.
$$
\end{lemma}

\subsection{Estimate of the mountain pass level}\lab{subsec:20230331}
In order to get a more precise information about the min-max level obtained by the Mountain Pass Theorem, we consider the Moser sequence of nonnegative functions
$$
\widetilde{m}_n(x)=\frac{1}{\sqrt{2\pi}}\left\{\begin{array}{lll}
(\log n)^{\frac{1}{2}}, & \text { if } & 0\leq|x| \leq \frac{1}{n}, \\
 \frac{\log\frac{1}{|x|}}{(\log n)^{\frac{1}{2}}}, & \text { if } & \frac{1}{n} \leq|x| \leq 1, \\
0, & \text { if } & |x| \geq 1 .
\end{array}\right.
$$
One has $\widetilde{m}_n \in H^1(\mathbb{R}^2)$, the support of $\widetilde{m}_n(x)$ is $B_1(0)$ and $\|\nabla \widetilde{m}_n\|_{2}^{2}=1$. By straightforward calculations,
\begin{align}\label{20230330e12}
\|\widetilde{m}_n\|_{2}^{2}&=\frac{\log n}{2n^2}+\frac{1}{\log n}\int_{\frac{1}{n}}^{1}\rho\log^2\frac{1}{\rho}\mathrm{d}\rho\nonumber\\
&=\frac{\log n}{2n^2}+\frac{1}{\log n}\big(\frac{1}{4}-\frac{1}{4n^2}-\frac{\log n}{2n^2}-\frac{\log^2 n}{2n^2}\big)\\
&=\frac{1}{4\log n}+{\rm{o}}(\frac{1}{\log n})\nonumber.
\end{align}
For any $n$ and $t$, let $\tau_n(t)$ be such that $u_c(\tau_n x)+t\widetilde{m}_n(\tau_n x)\in S_{r,c}$, that is $\|u_c(\tau_n x)+t\widetilde{m}_n(\tau_n x)\|_{2}^{2}=\tau_{n}^{-2}\|u_c+t\widetilde{m}_n\|_{2}^{2}=c^2$, which implies
\begin{equation}\label{tau_n}
\tau_{n}^{2}(t)=\frac{\|u_c+t\widetilde{m}_n\|_{2}^{2}}{c^2}= 1+\frac{t^{2}}{c^2}\|\widetilde{m}_n\|_{2}^{2}+\frac{2t}{c^2}\int_{\R^2}u_c\widetilde{m}_n\mathrm{d}x.
\end{equation}
By H\"{o}lder's inequality
\begin{align}\label{20230330e9}
&0\leq\Big|\int_{\R^2}u_c\widetilde{m}_n\mathrm{d}x\Big|\leq\int_{\R^2}\big|u_c\widetilde{m}_n\big|\mathrm{d}x\leq\big(\int_{B_1}|u_c|^2\mathrm{d}x\big)^{\frac{1}{2}}\big(\int_{B_1}|\widetilde{m}_n|^2\mathrm{d}x\big)^{\frac{1}{2}}\nonumber\\
&\leq \|u_c\|_{\infty}\sqrt \pi\big(\frac{1}{4\log n}-\frac{1}{2n^2}-\frac{1}{4n^2\log n}\big)^{\frac{1}{2}}\\
&=\frac{\|u_c\|_{\infty}\sqrt \pi }{2\sqrt{\log n}}+{\rm{o}}\big(\frac{1}{\log n}\big)\nonumber,
\end{align}
so that, by \eqref{20230330e9}
$$0\leq \int_{\R^2}u_c\widetilde{m}_n\mathrm{d}x= {\rm{O}}\left(\frac{1}{\sqrt{\log n}}\right).$$
Let us now consider the $L^2$-norm of the gradient of the function $u_c(\tau_n x)+t\widetilde{m}_n(\tau_n x)$:
$$
\|\nabla\big(u_c(\tau_n x)+t\widetilde{m}_n(\tau_n x)\big)\|_{2}^{2}=\|\nabla(u_c+t\widetilde{m}_n)\|_{2}^{2}=\|\nabla u_c\|_{2}^{2}+t^2+2t\int_{\mathbb{R}^2}\nabla u_c\nabla\widetilde{m}_n\mathrm{d}x,
$$
where, recalling that $u_c$ is a positive and radially decreasing function,
\begin{align}\label{20230330e11}
0\leq&\int_{\mathbb{R}^2}\nabla u_c\nabla\widetilde{m}_n\mathrm{d}x=\int_{\frac{1}{n}\leq|x|\leq 1}\nabla u_c\nabla\widetilde{m}_n\mathrm{d}x=\frac{\sqrt{2\pi}}{\sqrt{\log n}}\int_{\frac{1}{n}}^1u_c^{\prime}(\rho)\big(\log \frac{1}{\rho}\big)^\prime\rho\mathrm{d}\rho\nonumber\\
&=-\frac{\sqrt{2\pi}\big[u_c(1)-u_c(\frac 1n)\big]}{\sqrt{\log n}}\leq \frac{\sqrt{2\pi}\|u_c\|_{\infty}}{\sqrt{\log n}}.
\end{align}
Let $\gamma_n(t):=u_c(\tau_n x)+t\widetilde{m}_n(\tau_n x)\in S_{r,c}$ due to the definition \eqref{tau_n} of $\tau_n(t)$. For $n$ fixed, it holds that $\gamma_n(0)=u_c$ and
\begin{align}\label{20230330e1}
g_n(t):&=I(\gamma_n(t))=I(u_c(\tau_nx)+t\widetilde{m}_n(\tau_n x))\notag\\
&=\frac{1}{2}\|\nabla\big(u_c(\tau_n x)+t\widetilde{m}_n(\tau_n x)\big)\|_{2}^{2}-\int_{\mathbb{R}^2}F\big(u_c(\tau_n x)+t\widetilde{m}_n(\tau_n x)\big)\mathrm{d}x\notag\\
&=\frac{1}{2}\|\nabla(u_c+t\widetilde{m}_n)\|_{2}^{2}-\frac{1}{\tau_n^2}\int_{\R^2}F\big(u_c+t\widetilde{m}_n\big)\mathrm{d}x \nonumber\\
&=\frac{1}{2}\|\nabla u_c\|_{2}^{2}+\frac{t^{2}}{2}+t\int_{\mathbb{R}^2}\nabla u_c\nabla\widetilde{m}_n\mathrm{d}x-\frac{1}{\tau_n^2}\int_{\R^2}F\big(u_c+t\widetilde{m}_n\big)\mathrm{d}x,
\end{align}
so that $g_n(t)<2m_c$ for $t>0$ large enough and, as a consequence, $\gamma_n(t)\in\Gamma_c$.
Arguing as in  the proof of Lemma \ref{20230307-lemma-3}, we have that
$\max\limits_{t\geq 0}g_n(t)>0$ is attained at some $t_n>0$, for any fixed $n\in\N$.
Furthermore,
\begin{align}\label{20230330e2}
M_c\leq\max\limits_{t\geq 0}g_{n}(t)=\max\limits_{t\geq 0}I(\gamma_n(t)).
\end{align}
The following technical Lemma will be used in the proof of the estimate of the  mountain pass level.
\begin{lemma}\label{20230330-Lemma-1}
For any  $\delta>0$,~$t\in[0,T]$ there exists $R_\delta>0$ such that for any $s\geq R_\delta$,
$$
F(t+s)-F(t)\geq\Big(\frac{\beta_0-\delta}{2\alpha_0}\Big)\frac{e^{\alpha_0(t+s)^2}}{s^2}.
$$
\end{lemma}
\begin{proof}
Suppose by contradiction that there exists $\delta_0>0$,~$t_n\in[0,T]$ and $s_n\to\infty$ such that
$$
F(t_n+s_n)-F(t_n)<\Big(\frac{\beta_0-\delta_0}{2\alpha_0}\Big)\frac{e^{\alpha_0(t_n+s_n)^2}}{s_n^2},
$$
namely
\begin{align}\label{20230330e3}
F(t_n+s_n)<\Big(\frac{\beta_0-\delta_0}{2\alpha_0}\Big)\frac{e^{\alpha_0(t_n+s_n)^2}}{s_n^2}+F(t_n).
\end{align}
Using hypothesis $\mathbf{(f_4)}$, it holds
$$
\lim\limits_{s\to+\infty}\frac{F(s)}{s^{-2}e^{\alpha_0s^2}}=\lim\limits_{s\to+\infty}\frac{f(s)s}{2\alpha_0e^{\alpha_0s^2}}\geq\frac{\beta_0}{2\alpha_0}.
$$
Hence, for any $\delta>0$, there exists $R_\delta>0$ such that for $s\geq R_\delta$,
\begin{align*}
f(s)s\geq(\beta_0-\delta)e^{\alpha_0s^2},\quad F(s)s^2\geq\frac{\beta_0-\delta}{2\alpha_0}e^{\alpha_0s^2}
\end{align*}
and, in turns
\begin{align}\label{20230330e5}
F(t_n+s_n)\geq\Big(\frac{\beta_0-\frac{\delta_0}{2}}{2\alpha_0}\Big)\frac{e^{\alpha_0(t_n+s_n)^2}}{(t_n+s_n)^2}\geq\Big(\frac{\beta_0-\frac{5\delta_0}{8}}{2\alpha_0}\Big)\frac{e^{\alpha_0(t_n+s_n)^2}}{s_n^2}.
\end{align}
Combine \eqref{20230330e3} and \eqref{20230330e5} to obtain that
$$
F(t_n)\geq\big(\frac{3\delta_0}{16\alpha_0}\big)\frac{e^{\alpha_0(t_n+s_n)^2}}{s_n^2}\to+\infty\quad\text{as}\quad s_n\to\infty,
$$
which is false, since  $\displaystyle\sup_n F(t_n)\leq \max_{t\in [0,T]}F(t)<+\infty$.
\end{proof}

\noindent We are now ready to prove the following estimate for the mountain pass level.
\begin{proposition}\label{20230315-lemma-2}
$\max\limits_{t\geq 0}g_{n}(t)<m_c+\frac{2\pi}{\alpha_0}$ for some $n\in \N$.
\end{proposition}
\begin{proof}
Suppose by contradiction that for all $n\in\N$ one has
\beq\lab{eq:20221225-we1}
\max\limits_{t\geq 0}g_{n}(t)\geq m_c+\frac{2\pi}{\alpha_0}.
\eeq
Since the maximum is attained at some $t_n>0$, we have
$g_n(t_n)=\max\limits_{t\geq 0}g_n(t)$ and
\begin{align*}
g_n(t_n)=\frac{1}{2}\|\nabla(u_c+t_n\widetilde{m}_n)\|_{2}^{2}-\frac{1}{\tau_n^2}\int_{\R^2}F\big(u_c+t_n\widetilde{m}_n\big)\mathrm{d}x\geq m_c+\frac{2\pi}{\alpha_0},
\end{align*}
that is,
\begin{align*}
g_n(t_n)&=\left(\frac{1}{2}\|\nabla u_c\|_{2}^{2}-\int_{\R^2}F(u_c)\mathrm{d}x\right)+\frac{t_{n}^{2}}{2}+t_n\int_{\R^2}\nabla u_c \nabla \widetilde{m}_n\mathrm{d}x \nonumber \\
& -\frac{1}{\tau_n^2}\int_{\R^2}\big[F\big(u_c+t_n\widetilde{m}_n\big)-F(u_c)\big]\mathrm{d}x+(1-\frac{1}{\tau_n^2})\int_{\R^2}F(u_c)\mathrm{d}x\geq m_c+\frac{2\pi}{\alpha_0}.
\end{align*}
Since $\frac{1}{2}\|\nabla u_c\|_{2}^{2}-\int_{\R^2}F(u_c)\mathrm{d}x=m_c$ we obtain
\begin{align}\label{20230331e1}
&\frac{t_{n}^{2}}{2}+t_n\int_{\R^2}\nabla u_c \nabla \widetilde{m}_n\mathrm{d}x+(1-\frac{1}{\tau_n^2})\int_{\R^2}F(u_c)\mathrm{d}x-\frac{2\pi}{\alpha_0}\nonumber \\
&\quad \geq\frac{1}{\tau_n^2}\int_{\R^2}\big[F\big(u_c+t_n\widetilde{m}_n\big)-F(u_c)\big]\mathrm{d}x.
\end{align}

\noindent We \textit{claim} that $\{t_n\}$ is bounded. Indeed if not, up to a subsequence, we may assume that $t_n\rightarrow +\infty$. Combining the definition of $\tau_n$, \eqref{tau_n}, with the estimates \eqref{20230330e12}, \eqref{20230330e9}, \eqref{20230330e11} and \eqref{20230331e1} and applying  Lemma \ref{20230330-Lemma-1}, we have
\begin{align*}
&t_n^4\geq \tau_n^2\left(\frac{t_{n}^{2}}{2}+t_n{\rm{O}}\left(\frac 1{\sqrt{\log n} }\right)+{\rm{O}}\left(\frac 1{\sqrt{\log n} }\right)-\frac{2\pi}{\alpha_0}\right)\geq\int_{\R^2}\big[F\big(u_c+t_n\widetilde{m}_n\big)-F(u_c)\big]\mathrm{d}x\\
&\geq\int_{|x|\leq \frac{1}{n}}\Big(\frac{\beta_0-\delta}{2\alpha_0}\Big)\frac{e^{\alpha_0(u_c+t_n\widetilde{m}_n)^2}}{(t_n\widetilde{m}_n)^2}\mathrm{d}x
=\Big(\frac{\pi(\beta_0-\delta)}{2\alpha_0}\Big)\frac{e^{\alpha_0t_n^2\frac{1}{2\pi}\log n}}{t_n^2\frac{1}{2\pi}(\log n)n^2}\geq\Big(\frac{\pi^2(\beta_0-\delta)}{\alpha_0}\Big)\frac{e^{\alpha_0t_n^2}}{t_n^2}.
\end{align*}
Hence, we deduce that
$$
1\geq\Big(\frac{\pi^2(\beta_0-\delta)}{\alpha_0}\Big)\frac{e^{\alpha_0t_n^2}}{t_n^6}\to +\infty \quad \hbox{as } n\to +\infty,
$$
since $t_n\rightarrow +\infty$, which is a contradiction and thus the {\bf claim} is proved.

\noindent Now, up to a subsequence, we may assume that $\lim\limits_{n\to\infty}t_n=l<+\infty$.
Recalling \eqref{20230331e1} and combining the definition of $\tau_n$, \eqref{tau_n},  with the estimates \eqref{20230330e12}, \eqref{20230330e9}, \eqref{20230330e11} and \eqref{20230331e1}, we have
\begin{align*}
\frac{t_{n}^{2}}{2}&\geq\frac{2\pi}{\alpha_0}+{\rm{O}}\left(\frac 1{\sqrt{\log n}}\right)+\frac{1}{\tau_n^2}\int_{\R^2}\big[F\big(u_c+t_n\widetilde{m}_n\big)-F(u_c)\big]\mathrm{d}x+{\rm{O}}\left(\frac 1{\sqrt{\log n}}\right)\int_{\R^2}F(u_c)\mathrm{d}x\\
&\geq\frac{2\pi}{\alpha_0}+{\rm{O}}\left(\frac 1{\sqrt{\log n}}\right),
\end{align*}
which implies directly
\begin{align}\label{20230331e2}
\lim\limits_{n\to\infty}t_n=l\geq\sqrt{\frac{4\pi}{\alpha_0}}.
\end{align}

\noindent On the other hand, since $\max\limits_{t\geq0}g_n(t)$ is attained at $t_n>0$,
\begin{align}\label{20230404e2}
0=g_{n}^{\prime}(t_n)=\int_{\R^2}\nabla u_c \nabla \widetilde{m}_n\mathrm{d}x+t_n+\frac{2\tau_n^\prime}{\tau_n^3}\int_{\R^2}F\big(u_c+t_n\widetilde{m}_n\big)\mathrm{d}x-\frac{1}{\tau_{n}^{2}}\int_{\R^2}f(u_c+t_n\widetilde{m}_n)\widetilde{m}_n\mathrm{d}x.
\end{align}
Then we have
\begin{align}\label{g'=0}
&t_n^2+t_n\int_{\mathbb R^2} \nabla u_c\nabla \widetilde{m}_n\mathrm{d}x+2t_n\frac{\tau_n'(t_n)}{\tau_n^3(t_n)}\int_{\mathbb R^2}F(u_c+t_n\widetilde{m}_n)\mathrm{d}x=\frac{1}{\tau_n^2}\int_{\mathbb R^2}f(u_c+t_n \widetilde{m}_n)t_n\widetilde{m}_n\mathrm{d}x.
\end{align}
From \eqref{g'=0} we can first prove that
	\begin{equation}\label{Fbdd}
\int_{\mathbb R^2}F(u_c+t_n\widetilde{m}_n)\mathrm{d}x
	\end{equation}
is uniformly bounded. Indeed, by the definition of $\tau_n$, \eqref{tau_n}, and  \eqref{20230330e9}
$$
1\leq\tau_n(t_n)\leq 1+\frac{c_1}{\sqrt{\log n}}
$$
for some positive constant $c_1$, and
$$
0\leq 2\tau_n(t_n)\tau_n'(t_n)=\frac{2t_n}{c^2}\|\widetilde{m}_n\|_2^2+\frac{2}{c^2}\int_{\R^2}u_c\widetilde{m}_n \mathrm{d}x \leq  \frac {c_2}{\sqrt{\log n}},
$$
so that all the coefficients in front of the integrals in \eqref{g'=0} are bounded. Observe now that from assumption $\mathbf{(f_3)}$,  for any $\varepsilon>0$ there exists $s_\varepsilon$ such that $sf(s)>\frac 1{\varepsilon}F(s)$ for any $s>s_{\varepsilon}$. Moreover, since $\|u_c\|_{\infty}$ is bounded, choosing $s_\varepsilon$	big enough we have that $\frac{t_n\widetilde{m}_n}{u_c+t_n\widetilde{m}_n} \geq \frac 12$ if $u_c+t_n\widetilde{m}_n>s_\varepsilon$. Combining these properties with  \eqref{g'=0}  and assumption $\mathbf{(f_3)}$  we get, for some  $K>0$,
	\begin{align*}
&K+\int_{\R^2}F(u_c+t_n\widetilde{m}_n)\mathrm{d}x	\geq \frac 12\int_{\mathbb R^2}f(u_c+t_n \widetilde{m}_n)(u_c+t_n\widetilde{m}_n)\frac{t_n\widetilde{m}_n}{u_c+t_n\widetilde{m}_n}\mathrm{d}x\\
&\geq  \frac 1{4\varepsilon}\int\limits_{\{u_c+t_n \widetilde{m}_n>s_\varepsilon\}}F(u_c+t_n \widetilde{m}_n)\mathrm{d}x+ \frac 12\int\limits_{\{u_c+t_n \widetilde{m}_n\leq s_\varepsilon\}}f(u_c+t_n \widetilde{m}_n)t_n\widetilde{m}_n\mathrm{d}x\\
&\geq  \frac 1{4\varepsilon}\int_{\R^2}F(u_c+t_n \widetilde{m}_n)\mathrm{d}x- \frac 1{4\varepsilon}\int\limits_{\{u_c+t_n \widetilde{m}_n\leq s_\varepsilon\}}F(u_c+t_n \widetilde{m}_n)\mathrm{d}x\\
&\geq  \frac 1{4\varepsilon}\int_{\R^2}F(u_c+t_n \widetilde{m}_n)\mathrm{d}x- C_\varepsilon\int\limits_{\{u_c+t_n \widetilde{m}_n\leq s_\varepsilon\}}|u_c+t_n \widetilde{m}_n|^p\mathrm{d}x\\
&\geq  \frac 1{4\varepsilon}\int_{\R^2}F(u_c+t_n \widetilde{m}_n)\mathrm{d}x- C_\varepsilon\|u_c+t_n \widetilde{m}_n\|_p^p,
	\end{align*}
which by choosing  $\varepsilon=1/8$ gives directly \eqref{Fbdd}.

\noindent By assumption $\mathbf{(f_4)}$, and recalling that $u_c$ is a positive and radially decreasing function, for any $\delta>0$ we obtain
\begin{align*}
&t_n\tau_{n}^{2}\left(\int_{\R^2}\nabla u_c \nabla \widetilde{m}_n\,dx+t_n+\frac{2\tau_n^\prime}{\tau_n^3}\int_{\R^2}F\big(u_c+t_n\widetilde{m}_n\big)\mathrm{d}x\right)=\int_{\R^2}f(u_c+t_n\widetilde{m}_n)t_n\widetilde{m}_n\mathrm{d}x\\
&\geq\int_{0\leq|x|\leq\frac{1}{n}}f(u_c+t_n\widetilde{m}_n)t_n\widetilde{m}_n\mathrm{d}x=\int_{0\leq|x|\leq\frac{1}{n}}f(u_c+t_n \widetilde{m}_n)(u_c+t_n\widetilde{m}_n)\frac{t_n\widetilde{m}_n}{u_c+t_n\widetilde{m}_n}\mathrm{d}x
\\
&\geq \frac \pi{n^2}(\beta_0-2\delta)e^{\alpha_0(\|u_c\|_{\infty}+{\rm{o}}(1)+\frac{t_n}{\sqrt{2\pi}}\sqrt{\log n})^2}\frac{t_n\sqrt{\log n}}{\sqrt{2\pi }\|u_c\|_{\infty}+{\rm{o}}(1)+t_n(\sqrt{\log n})}\\
&\geq \frac\pi{2}(\beta_0-2\delta)e^{2\left(\frac{\alpha_0}{4\pi}t_n^2-1\right)\log n + {\rm{O}}(\sqrt{\log n} )},
\end{align*}
which implies that, up to a subsequence, $l\leq \sqrt{\frac{4\pi}{\alpha_0}}$.

\noindent Combining this with \eqref{20230331e2} yields
\begin{equation}\label{el}
\lim\limits_{n\to\infty}t_n=l=\sqrt{\frac{4\pi}{\alpha_0}}.
\end{equation}

\noindent From \eqref{g'=0} and assumption $\mathbf{(f_4)}$, for any $\delta >0$
\begin{align}\label{g'=01}
\nonumber&t^2_n+{\rm{O}}\left(\frac{1}{\sqrt{\log n}}\right)\geq t^2_n+t_n\int_{\R^2}\nabla u_c \nabla\widetilde{m}_n\mathrm{d}x+\frac{2t_n\tau_n^\prime}{\tau_n^3}\int_{\R^2}F\big(u_c+t_n\widetilde{m}_n\big)\mathrm{d}x\\
&\nonumber=\frac{1}{\tau_{n}^{2}}\int_{\R^2}f(u_c+t_n\widetilde{m}_n)(u_c+t_n\widetilde{m}_n)\frac{t_n\widetilde{m}_n}{u_c+t_n\widetilde{m}_n}\mathrm{d}x\\
&\nonumber\geq\frac{1}{\tau_{n}^{2}}\int_{B_{\frac 1n}}f(u_c+t_n\widetilde{m}_n)(u_c+t_n\widetilde{m}_n)\frac{t_n\widetilde{m}_n}{u_c+t_n\widetilde{m}_n}\mathrm{d}x \\
&\nonumber \geq\frac{\beta_0-\delta}{\tau_{n}^{2}}\int_{B_{\frac 1n}}e^{\alpha_0(u_c+t_n\widetilde{m}_n)^2}\frac{t_n\sqrt{\log n}}{\sqrt{2\pi}\|u_c\|_{\infty}+t_n\sqrt{\log n}}\mathrm{d}x\\
&\geq(\beta_0-\delta)\frac{\pi}{n^2\tau_{n}^{2}}e^{\alpha_0\left(\|u_c\|_\infty+{\rm{o}}(1)+t_n\frac{\sqrt{\log n}}{\sqrt{2\pi}}\right)^2}\frac{t_n\sqrt{\log n}}{\sqrt{2\pi}\|u_c\|_{\infty}+t_n\sqrt{\log n}}
\end{align}
for $n$ large enough. Combining the previous inequality with \eqref{el} we obtain
\begin{equation}\label{est_below}
{\rm{o}}(1)+\frac{4\pi}{\alpha_0}\geq (\beta_0-2\delta)\pi e^{\alpha_0\left(\|u_c\|_\infty+{\rm{o}}(1)+t_n\frac{\sqrt{\log n}}{\sqrt{2\pi}}\right)^2-2\log n}.
\end{equation}
We now \textit{claim} the following lower bound
\begin{equation}\label{claim}
t_n+{\rm{O}}\left(\frac 1{\log n}\right)\geq \sqrt{\frac {4\pi}{\alpha_0}}.
\end{equation}
Inserting this bound in the exponent on the right hand side of \eqref{est_below} yields
\begin{align*}
&\alpha_0\left(\|u_c\|_{\infty}+{\rm{o}}(1)+t_n\sqrt{\frac{\log n}{2\pi}}\right )^2-2\log n\geq  \alpha_0\left(\|u_c\|_{\infty}+{\rm{o}}(1)+\sqrt{\frac{2\log n}{\alpha_0}}\right )^2-2\log n\\
&=\alpha_0\|u_c\|_{\infty}^2+2\sqrt{\alpha_0}\|u_c\|_{\infty}\sqrt{2\log n}+{\rm{o}}(\sqrt{\log n})\to +\infty,
\end{align*}
which gives a contradiction.

\medskip

\noindent Let us now prove \eqref{claim}.  A key step here is the combination of the Pohozaev and of the Nehari manifolds, which give a fundamental identity that will enable us to exclude terms of order $\frac{1}{\sqrt{\log n}}$ when estimating  $t_n$.\newline
Since $u_c$ is a solution of problem \eqref{20230228e1}, \eqref{20230228e2} (which is also a positive, radially symmetric non-increasing function), we know that
$$
\|\nabla u_c\|_2^2+\lambda_c\int_{\R^2} u_c^2\mathrm{d}x-\int_{\R^2}f(u_c)u_c\mathrm{d}x=0.
$$
On the other hand, $\mathcal P(u_c)=0$, that is
$$
\|\nabla u_c\|_2^2+2\int_{\R^2}F(u_c)\mathrm{d}x-\int_{\R^2}f(u_c)u_c\mathrm{d}x=0,
$$
which gives
\begin{equation}\label{identity}
2\int_{\R^2}F(u_c)\mathrm{d}x=\lambda_c\int_{\R^2} u_c^2\mathrm{d}x=c^2\lambda_c.
\end{equation}
Moreover, by definition of weak solution we also have
\begin{align}\label{20230330e6}
t_n\int_{\R^2}\nabla u_c\nabla \widetilde{m}_n\mathrm{d}x +\lambda_ct_n\int_{\R^2}u_c\widetilde{m}_n\mathrm{d}x=\int_{\R^2}f(u_c)\cdot t_n\widetilde{m}_n\mathrm{d}x.
\end{align}
Combining $\mathbf{(f_5)}$ with \eqref{20230330e6}, we obtain
\begin{align}\label{20230529e1}
&\frac{1}{\tau_n^2}\int_{\R^2}\left[F\big(u_c+t_n\widetilde{m}_n\big)-F(u_c)\right]\mathrm{d}x=\frac{1}{\tau_n^2}\int_{\R^2}f\big(u_c+\theta t_n\widetilde{m}_n\big)t_n\widetilde{m}_n\mathrm{d}x\notag\\
&\geq\frac{1}{\tau_n^2}\int_{\R^2}f(u_c)t_n\widetilde{m}_n\mathrm{d}x=\frac{t_n}{\tau_n^2}\int_{\R^2}\left(\nabla u_c\nabla\widetilde{m}_n+\lambda_cu_c\widetilde{m}_n\right)\mathrm{d}x
\end{align}
which inserted into \eqref{20230331e1}, yields
\begin{align}\label{20230330e8}
&\frac{t^2_n}{2}+t_n\left(1-\frac{1}{\tau_n^2}\right)\int_{\R^2}\nabla u_c\nabla \widetilde{m}_n\mathrm{d}x+\left(1-\frac{1}{\tau_n^2}\right)\int_{\R^2}F(u_c)\mathrm{d}x\nonumber \\
&\geq \frac{\lambda_c}{\tau_n^2}t_n\int_{\R^2}u_c\widetilde{m}_n \mathrm{d}x +\frac{2\pi}{\alpha_0}.
\end{align}
On the other hand, from the definition \eqref{tau_n} of $\tau_n^2$,
$$
1-\frac 1{\tau^2_n}=\frac{1}{\tau_n^2}\left(\frac{t^2_n}{c^2}\|\widetilde{m}_n\|_2^2+2\frac{t_n}{c^2}\int_{\R^2}u_c\widetilde{m}_n\mathrm{d}x\right),
$$
so that from \eqref{20230330e8} we obtain
\begin{align*}
&\frac{t^2_n}{2}+t_n\left(1-\frac{1}{\tau_n^2}\right)\int_{\R^2}\nabla u_c\nabla \widetilde{m}_n\mathrm{d}x+\frac{t_n^2}{\tau_n^2c^2}\|\widetilde{m}_n\|_2^2\int_{\R^2}F(u_c)\mathrm{d}x\\
&+\frac{t_n}{c^2\tau_n^2}\int_{\R^2}u_c\widetilde{m}_n\mathrm{d}x\left(2\int_{\R^2}F(u_c)\,dx -\lambda_c\cdot c^2\right)\geq \frac{2\pi}{\alpha_0}.
\end{align*}
Due to the identity \eqref{identity},  the definition of $\tau_n^2$, \eqref{tau_n}, and the estimates \eqref{20230330e12}, \eqref{20230330e9} we finally obtain
$$
t_n^2+{\rm{O}}\left(\frac 1{\log n}\right)\geq \frac{4\pi}{\alpha_0},
$$
and, as a consequence, our claim \eqref{claim}.

\end{proof}

\subsection{Compactness}\lab{subsec:20230307-1}
Let us now prove the validity of the $(PSP)_{M_c}$-condition.
\begin{lemma}\label{20230307-lemma-2}
For any $c \in(0, c_0)$, suppose that
$$
M_c<m_c+\frac{2\pi}{\alpha_0}.
$$
Then the sequence obtained in Lemma \ref{lemma 3.3} is relatively compact in $H_r^{1}(\mathbb{R}^2)$.
\end{lemma}
\begin{proof}
Let $\{\bar{u}_n\} \subset H_r^{1}(\mathbb{R}^2)$ be given by Lemma \ref{lemma 3.3}. We proceed in three steps.

\medskip

\noindent {\bf{Step 1}}: $\{\bar{u}_n\} \subset H_r^{1}(\mathbb{R}^2)$ is bounded.

\noindent  Suppose by contradiction that $\|\nabla \bar{u}_n\|_{2}^{2}\to +\infty$. From $I(\bar{u}_n)\to M_c$ we immediately get
$$
\int_{\R^2}F(\bar u_n)\mathrm{d}x \to +\infty
$$
and then, from $\mathcal{P}(\bar{u}_n)\to 0$ ($\mathcal P$ defined in \eqref{Pmanifold}),
$$
\int_{\R^2}f(\bar u_n)\bar u_n\mathrm{d}x \to +\infty.
$$
Now, combining $I(\bar{u}_n)\to M_c$ and $\mathcal{P}(\bar{u}_n)\to 0$ gives
$$
4\int_{\R^2}F(\bar u_n)\mathrm{d}x-\int_{\R^2}f(\bar{u}_n)\bar{u}_n\mathrm{d}x+2M_c={\rm{o}}_n(1).
$$
 By assumption $(f_3)$,  there is $s_0$ such that $F(s)<\frac 18f(s)s$ for any $s>s_0$. Then
\begin{align*}
{\rm{o}}_n(1)=&4\int_{\R^2}F(\bar u_n)\mathrm{d}x-\int_{\R^2}f(\bar{u}_n)\bar{u}_n \mathrm{d}x+2M_c\\
&< 4\int_{\{\bar u_n \leq s_0\}}F(\bar u_n)\mathrm{d}x-\frac 12 \int_{\{\bar u_n >s_0\}}f(\bar{u}_n)\bar{u}_n\mathrm{d}x- \int_{\{\bar u_n \leq s_0\}}f(\bar{u}_n)\bar{u}_n \mathrm{d}x+2M_c\\
&\leq 4\int_{\{\bar u_n \leq s_0\}}F(\bar u_n)\mathrm{d}x-\frac 12 \int_{\R^2}f(\bar{u}_n)\bar{u}_n\mathrm{d}x+2M_c.
\end{align*}
From {\color{blue}$\mathbf{(f_1)}$}, since $p>2$ we easily get $F(s)\leq as^2$ for any $s\leq s_0$, for some positive constant $a$: combining with the previous inequality yields
\begin{align*}
{\rm{o}}_n(1)&\leq4\int_{\{\bar{u}_n\leq s_0\}}F(\bar u_n)\mathrm{d}x-\frac{1}{2}\int_{\R^2}f(\bar{u}_n)\bar{u}_n \mathrm{d}x+2M_c\\
&\leq 4a\|\bar u_n\|_2^2-\frac 12 \int_{\R^2}f(\bar{u}_n)\bar{u}_n \mathrm{d}x+2M_c\\
&=4ac^2+2M_c -\frac 12 \int_{\R^2}f(\bar{u}_n)\bar{u}_n \mathrm{d}x\to -\infty,
\end{align*}
 which is a contradiction, hence the sequence $\{\|\nabla \bar{u}_n\|_{2}^{2}\}$ stays bounded, as well as the sequences $\{\|\bar{u}_n\|_{H^{1}(\R^2)}\}$, $\{\int_{\R^2}F(\bar{u}_n)\mathrm{d}x\}$,~$\{\int_{\R^2}f(\bar{u}_n)\bar{u}_n\mathrm{d}x\}$ and $\{|\lambda_n|\}$. Thus, we may assume, passing to a subsequence if necessary, that $\bar{u}_n\rightharpoonup \bar{u}_c$ weakly in $H_{r}^{1}(\R^2)$ and $\lambda_n\to\bar{\lambda}_c$, as $n\to\infty$.

\medskip

\noindent {\bf{Step 2}}: $\{\bar{u}_n\} \subset H_r^{1}(\R^2)$ has a non-trivial weak limit.

\noindent  Since $\bar{u}_n\rightharpoonup \bar{u}_c$ weakly in $H_{r}^{1}(\R^2)$, by the compact embedding of $H_r^{1}(\R^2)$ into $L^q(\R^2)$, for any $q>2$, there exists a $\bar{u}_c \in H_r^{1}(\R^2)$ such that, up to a subsequence, $\bar{u}_n \rightarrow \bar{u}_c$ strongly in $L^q(\R^2)$.

\noindent Let us first prove  that $\bar{u}_c\neq 0$. If not, by the radial compact embedding result, we have that $\|\bar{u}_n\|_{p}=o_n(1)$. By \eqref{eq:20230425-we1}, for any $\varepsilon>0$, we can find some $C_\varepsilon>0$ such that
\begin{align}\label{eq:20230222-e1}
\lim\limits_{n\to\infty}\int_{\R^2}F(\bar{u}_n)\mathrm{d}x\leq & \varepsilon\lim\limits_{n\to\infty}\int_{\R^2}f(\bar{u}_n)\bar{u}_n\mathrm{d}x+C_\varepsilon\lim\limits_{n\to\infty}\|\bar{u}_n\|_{p}^{p}\\
\leq &\varepsilon\lim\limits_{n\to\infty}\int_{\R^2}f(\bar{u}_n)\bar{u}_n\mathrm{d}x+o_n(1).\nonumber
\end{align}
Since $\{\int_{\R^2}f(\bar{u}_n)\bar{u}_n\mathrm{d}x\}$ is bounded we conclude that
\begin{align*}
\lim\limits_{n\to\infty}\int_{\R^2}F(\bar{u}_n)\mathrm{d}x=0.
\end{align*}
As a consequence
\begin{align}\label{3.32}
M_c=\frac{1}{2}\lim\limits_{n\to\infty}\|\nabla \bar{u}_n\|_{2}^{2}
\end{align}
and
\begin{align*}
\lim\limits_{n\to\infty}\|\nabla \bar{u}_n\|_{2}^{2}<2m_c+\frac{4\pi}{\alpha_0}<\frac{4\pi}{\alpha_0},
\end{align*}
since $m_c<0$. By Corollary \ref{corollary 3.2}, we then have
\begin{align*}
\lim\limits_{n\to\infty}\int_{\R^2}f(\bar{u}_n)\bar{u}_n\mathrm{d}x=0.
\end{align*}
Combining the above with $\mathcal{P}(\bar{u}_n)\to 0$, we obtain that $\|\nabla \bar{u}_n\|_{2}^{2}\to 0$. And thus \eqref{3.32} implies $M_c=0$, which is in contradiction since $M_c>0$ by Lemma \ref{20230307-lemma-3}. Hence $\bar{u}_c\neq 0$.

\medskip

\noindent {\bf{Step 3}}: $\bar{u}_n\to \bar{u}_c$, as $n\to\infty$, strongly in $\subset H_r^{1}(\R^2)$.

\noindent Define the functional $\varphi: H_{r}^{1}(\R^2)\to\R$ as follows
$$
\varphi(\bar{u})=\frac{1}{2}\|\bar{u}\|_{2}^{2}.
$$
By $\{\bar{u}_n\}\subset S_{r,c}$ and $I\big|_{S_{r,c}}^{\prime}(\bar{u}_n)\to 0$,  there exists $\{\lambda_n\}\subset\R$ such that
\begin{align}\label{3.30}
\|I^{\prime}(\bar{u}_n)+\lambda_n\varphi^{\prime}(\bar{u}_n)\|\to 0.
\end{align}
Since $\bar{u}_n\rightarrow \bar{u}_c$ in $L^{q}(\R^2)$ for any $q>2$, the sequence $\{|\bar u_n|^p\}$ is uniformly integrable by de la Vall\'ee-Poussin Theorem. Moreover,  $\{f(\bar{u}_n)\bar{u}_n\}$ is bounded in $L^1(\R^2)$,  so that from \eqref{eq:20230222-e1} one has that $\int_{\R^2}F(\bar{u}_n)\mathrm{d}x$ is uniformly bounded and possesses the uniform integrability condition by the arbitrariness of $\varepsilon$. So, by  Vitali's convergence theorem, we obtain
\begin{align}\label{3.36}
\lim\limits_{n\to\infty}\int_{\R^2}F(\bar{u}_n)\mathrm{d}x=\int_{\R^2}\lim\limits_{n\to\infty}F(\bar{u}_n)\mathrm{d}x=\int_{\R^2}F(\bar{u}_c)\mathrm{d}x.
\end{align}
Let $v_n=\bar{u}_n-\bar{u}_c$. By \eqref{3.36}, we have
\begin{align*}
\frac 12\|\nabla v_n\|_2^2&=\frac 12\|\nabla \bar u_n\|_2^2+\frac 12 \|\nabla \bar{u}_c\|_2^2-\int_{\R^2}\nabla \bar u_n \nabla \bar{u}_c \mathrm{d}x\\
&= \frac 12\|\nabla \bar u_n\|_2^2-\frac 12 \|\nabla\bar{u}_c\|_2^2+ {\rm{o}}_n(1)=M_c+\int_{\R^2}F(\bar u_n) \mathrm{d}x-\frac 12 \|\nabla \bar{u}_c\|_2^2+ {\rm{o}}_n(1)\\
&=M_c+\int_{\R^2}F( u_c) \mathrm{d}x-\frac 12 \|\nabla \bar{u}_c\|_2^2+ {\rm{o}}_n(1),
\end{align*}
so that
\begin{align}\label{3.37}
M_c=\frac{1}{2}(\lim\limits_{n\to\infty}\|\nabla v_n\|_{2}^{2}+\|\nabla \bar{u}_c\|_{2}^{2})-\int_{\R^2}F(\bar{u}_c)\mathrm{d}x.
\end{align}
Using \eqref{3.30}, \eqref{3.36} and $\mathcal{P}(\bar{u}_n)\to 0$, by applying Lemma 2.1 in \cite{Figueiredo1995}  we have
\begin{align}\label{3.38}
-\Delta \bar{u}_c+\bar{\lambda}_c\bar{u}_c=f(\bar{u}_c),
\end{align}
where
\begin{align}\label{3.39}
\bar{\lambda}_c=&\lim_{n\rightarrow \infty}\lambda_n
=\lim_{n\rightarrow \infty}-\frac{I'(\bar{u}_n)\bar{u}_n}{\varphi'(\bar{u}_n)\bar{u}_n}\nonumber\\
=&\lim_{n\rightarrow \infty}-\frac{\|\nabla \bar{u}_n\|_2^2 -\int_{\R^2}f(\bar{u}_n)\bar{u}_n \mathrm{d}x}{c^2}\nonumber\\
=&\lim_{n\rightarrow \infty}\frac{2\int_{\R^2}F(\bar{u}_n)\mathrm{d}x}{c^2}
=\frac{2}{c^2}\int_{\R^2}F(\bar{u}_c)\mathrm{d}x>0.
\end{align}
On the other hand, since $\bar{u}_c$ solves \eqref{3.38}, we also have the following Pohozaev identity
\beq\lab{eq:20221226-ze2}
\bar{\lambda}_c\|\bar{u}_c\|_{2}^{2}=2\int_{\R^2}F(\bar{u}_c)\mathrm{d}x.
\eeq
It follows from \eqref{3.39} and \eqref{eq:20221226-ze2} that  $\bar{u}_c\in S_{r,c}$ and thus $\bar{u}_n\to\bar{u}_c$ in $L^2(\R^2)$ and
\begin{align}\label{3.40}
\|\nabla \bar{u}_c\|_{2}^{2}+2\int_{\R^2}F(\bar{u}_c)\mathrm{d}x-\int_{\R^2}f(\bar{u}_c)\bar{u}_c\mathrm{d}x=0.
\end{align}
From \eqref{3.37} we also have
\begin{align}\label{20230801e1}
m_c+\frac{2\pi}{\alpha_0}>M_c&=\Big(\frac{1}{2}\|\nabla \bar{u}_c\|_{2}^{2}-\int_{\R^2}F(\bar{u}_c)\mathrm{d}x\Big)+\frac{1}{2}\lim\limits_{n\to\infty}\|\nabla v_n\|_{2}^{2}\nonumber\\
&\geq m_c+\frac{1}{2}\lim\limits_{n\to\infty}\|\nabla v_n\|_{2}^{2},
\end{align}
so that
\begin{align*}
\lim\limits_{n\to\infty}\|\nabla v_n\|_{2}^{2}<\frac{4\pi}{\alpha_0}.
\end{align*}
According to Corollary \ref{corollary 3.2}, we deduce
\begin{align}\label{3.44}
\lim\limits_{n\to\infty}\int_{\R^2}f(\bar{u}_n)\bar{u}_n\mathrm{d}x=\int_{\R^2}f(\bar{u}_c)\bar{u}_c\mathrm{d}x.
\end{align}
Now, combining \eqref{3.36},\eqref{3.40},\eqref{3.44} and the fact of $\mathcal{P}(\bar{u}_n)\to 0$, we have that $\|\nabla \bar{u}_n\|_{2}^{2}\to\|\nabla \bar{u}_c\|_{2}^{2}$. Hence, $\bar{u}_n\to \bar{u}_c\neq0$ in $H_{r}^{1}(\R^2)$.
\end{proof}
\vspace*{0.2cm}

\noindent{\bf Proof of Theorem \ref{theorem 1.2}.}
By Lemma \ref{lemma 3.3}, there exists a $(PSP)_{M_c}$-sequence for $I$ constrained on $S_{r,c}$. It follows from Lemma \ref{20230307-lemma-3} that $M_c>0$. Furthermore, by Proposition \ref{20230315-lemma-2} and \eqref{20230330e2} we have $M_c<m_c+\frac{2\pi}{\alpha_0}$. Then, we can apply the compactness Lemma \ref{20230307-lemma-2} to find a normalized mountain pass critical point of $\left.I\right|_{S_{r,c}}$, namely $\bar{u}_c\in H_{r}^{1}(\R^2)$ together with some $\bar{\lambda}_c\in \R$ such that $(\bar{u}_c,\bar{\lambda}_c)$ is a normalized mountain pass type solution to \eqref{20230228e1}-\eqref{20230228e2}. In particular, it follows from \eqref{3.39} that $\bar{\lambda}_c>0$. Testing \eqref{20230228e1} by $\bar{u}_{c,-}:=-\min\{\bar{u}_c, 0\}$, one can see that $\bar{u}_{c,-}=0$ and thus $\bar{u}_c\geq 0$ in $\R^2$. Finally, by the strong maximum principle, we conclude that $\bar{u}_c$ is positive.
\hfill$\Box$

\subsection{Asymptotic behavior of the mountain pass solution as $c\to0^+$}\lab{sec:proofs-3} In this subsection, we are aim to answer the question ({\bf{Q4}}).

%

\vskip 0.2in
\noindent
{\bf Proof of Theorem \ref{theorem 1.3}:}
Let $c\to0^+$, and $(\bar u_c,\bar \lambda_c)$ be the mountain pass type critical point of $I\big|_{S_{r,c}}$ given in Theorem \ref{theorem 1.2}, then we have $\mathcal{P}(\bar u_c)=0$ and $0<I(\bar u_c)=M_c<m_c+\frac{2\pi}{\alpha_0}$. Applying the same argument as in Step 1 of the proof of Lemma \ref{20230307-lemma-2}, we can prove that $\{\bar u_c\}_{0<c<c_0}$ is bounded in $H_{r}^{1}(\R^2)$. Since $\|\bar u_c\|_2^2\to0$, by the Gagliardo-Nirenberg inequality, $\|\bar u_c\|_{p}^{p}\to0$, $\forall~ 2<p<+\infty$. By \eqref{eq:20230425-we1}, for any $\varepsilon>0$, we can find some $C_\varepsilon>0$ such that
\begin{align*}
\lim\limits_{c\to0^+}\int_{\R^2}F(\bar{u}_c)\mathrm{d}x\leq & \varepsilon\lim\limits_{c\to0^+}\int_{\R^2}f(\bar{u}_c)\bar{u}_c\mathrm{d}x+C_\varepsilon\lim\limits_{c\to0^+}\|\bar{u}_c\|_{p}^{p}\\
\leq &\varepsilon\lim\limits_{c\to0^+}\int_{\R^2}f(\bar{u}_c)\bar{u}_c\mathrm{d}x+o_n(1).\nonumber
\end{align*}
Since $\{\int_{\R^2}f(\bar{u}_c)\bar{u}_c\mathrm{d}x\}$ is bounded, by the arbitrariness of $\varepsilon$, we conclude that
\begin{align}\lab{eq:20230828-xbe2}
\lim\limits_{c\to0^+}\int_{\R^2}F(\bar{u}_c)\mathrm{d}x=0.
\end{align}

\noindent Combining this with $\mathcal{P}(\bar u_c)=0$, we have that
\begin{align}\label{20230802e1}
\|\nabla \bar u_c\|_2^2=\int_{\R^2}f(\bar u_c)\bar u_c\mathrm{d}x+o(1).
\end{align}
Note that $\|\nabla \bar u_c\|_2^2$ is bounded away from $0$: indeed, if $\|\nabla \bar u_{c_n}\|_2^2\to0$ for some sequence $c_n\rightarrow 0^+$, then $M_{c_n}=I(\bar u_{c_n})\to0$, which contradicts  $M_c\geq \tau_0>0, \forall~c\in (0,c_0)$ (see Lemma \ref{20230307-lemma-3}).
On one hand, by $I(\bar u_c)=M_c<m_c+\frac{2\pi}{\alpha_0}$ and $m_c\rightarrow 0$ as $c\rightarrow 0^+$, we conclude that
$\limsup\limits_{c\rightarrow 0^+}\|\nabla \bar u_c\|_2^2\leq\frac{4\pi}{\alpha_0}$.

\noindent We \textit{claim} that $\displaystyle\lim_{c\rightarrow 0^+}\|\nabla \bar u_c\|_2^2=\frac{4\pi}{\alpha_0}$. Indeed if not,
there exists a sequence $c_n\rightarrow 0^+$ such that $\displaystyle \lim_{n\rightarrow \infty}\|\nabla \bar{u}_{c_n}\|_2^2< \frac{4\pi}{\alpha_0}$. Noting that $\bar{u}_{c_n}\rightharpoonup 0$ in $H^1(\R^2)$,  by Corollary \ref{corollary 3.2} we deduce that
\begin{align}\label{20230802e2}
\int_{\R^2}f(\bar{u}_{c_n})\bar{u}_{c_n}\mathrm{d}x\to0~~\text{as}~n\to \infty.
\end{align}
Combining with \eqref{20230802e1} and \eqref{20230802e2}, we conclude that $\|\nabla \bar{u}_{c_n}\|_2^2\rightarrow 0$, which is a contradiction to the fact that $\|\nabla \bar{u}_{c_n}\|_2^2\geq \tau_0 >0$. The \textit{claim} is proved.

\noindent Finally, by \eqref{20230802e1} we have that $\lim\limits_{c\to 0^+}\int_{\R^2}f(\bar u_c)\bar u_c\mathrm{d}x=\frac{4\pi}{\alpha_0}$ and by \eqref{eq:20230828-xbe2} we conclude that $M_c\to\frac{2\pi}{\alpha_0}$, as $c\to 0^+$.
\hfill$\Box$


\begin{thebibliography}{10}

\bibitem{Adachi2000}
S.~Adachi and K.~Tanaka.
\newblock Trudinger type inequalities in $\mathbb{R}^{N}$ and their best
  exponents.
\newblock {\em Proc. Amer. Math. Soc.}, {\bf 128}(7):2051--2057, 2000.

\bibitem{Adimurthi2007}
A.~Adimurthi and K.~Sandeep.
\newblock A singular {M}oser-{T}rudinger embedding and its applications.
\newblock {\em NoDEA Nonlinear Differential Equations Appl.}, {\bf 13}(5-6):585--603,2007.

\bibitem{Akahori2013}
T.~Akahori, S.~Ibrahim, H.~Kikuchi and H.~Nawa.
\newblock Existence of a ground state and scattering for a nonlinear
  {S}chr\"{o}dinger equation with critical growth.
\newblock {\em Selecta Math. (N.S.)}, {\bf 19}(2):545--609, 2013.

\bibitem{Alves2022}
C.~O. Alves, C.~Ji and O.~H. Miyagaki.
\newblock Normalized solutions for a {S}chr\"{o}dinger equation with critical
  growth in $\mathbb{R}^{N}$.
\newblock {\em Calc. Var. Partial Differential Equations}, {\bf 61}(1):Paper No. 18, 24 pp, 2022.


\bibitem{Araujo2018}
L.~A.~de~Araujo and M.~Montenegro.
\newblock Existence of solution for a nonvariational elliptic system with
exponential growth in dimension two.
\newblock {\em J. Differential Equations}, {\bf 264}(3):2270--2286, 2018.

\bibitem{Araujo2022}
 L.~A. de~Araujo and F.~O. Faria.
\newblock Existence, nonexistence, and asymptotic behavior of solutions for
  {$N$}-{L}aplacian equations involving critical exponential growth in the
  whole {$\Bbb R^N$}.
\newblock {\em Math. Ann.}, {\bf 384}(3-4):1469--1507, 2022.

\bibitem{Bartsch2013}
T.~Bartsch and S.~De~Valeriola.
\newblock Normalized solutions of nonlinear {S}chr{\"o}dinger equations.
\newblock {\em Arch. Math.(Basel)}, {\bf 100}(1):75--83, 2013.

\bibitem{Bartsch2016}
T.~Bartsch, N.~Soave and L.~Jeanjean.
\newblock Normalized solutions for a system of coupled cubic {S}chr\"{o}dinger
equations on $\mathbb{R}^3$.
\newblock {\em J. Math. Pures Appl.}, {\bf 106}(4):583--614, 2016.


\bibitem{Bartsch2017}
T.~Bartsch and N.~Soave.
\newblock A natural constraint approach to normalized solutions of nonlinear
  {S}chr{\"o}dinger equations and systems.
\newblock {\em J. Funct. Anal.}, {\bf 272}(12):4998--5037, 2017.


\bibitem{Bartsch2021}
T.~Bartsch, X.~X.~Zhong and W.~M.~Zou.
\newblock Normalized solutions for a coupled {S}chr\"{o}dinger system.
\newblock {\em Math. Ann.}, {\bf 380}:1713--1740, 2021.

\bibitem{Bieganowski2021}
B.~Bieganowski and J.~Mederski.
\newblock Normalized ground states of the nonlinear {S}chr{\"o}dinger equation
  with at least mass critical growth.
\newblock {\em J. Funct. Anal.}, {\bf 280}(11):108989, 26 pp,2021.


\bibitem{Cao1992}
D.~M.~Cao.
\newblock Nontrivial solution of semilinear elliptic equation with critical
  exponent in $\mathbb{R}^2$.
\newblock {\em Comm. Partial Differential Equations}, {\bf 17}(3-4):407--435, 1992.

\bibitem{Cassani2014}
D.~Cassani, F.~Sani and C.~Tarsi.
\newblock Equivalent Moser type inequalities in  $\mathbb{R}^2$ and the zero mass case.
\newblock {\em J. Funct. Anal.}, {\bf 267}(11):4236--4263,2014.


\bibitem{Cassani2021}
D.~Cassani and C.~Tarsi.
\newblock Schr{\"o}dinger-Newton equations in dimension two via a Pohozaev-Trudinger log-weighted inequality.
\newblock {\em Calc. Var. Partial Differential Equations}, {\bf 60}(5): Paper No. 197, 31 pp, 2021.



\bibitem{Cazenave2003}
T.~Cazenave.
\newblock {\em Semilinear {S}chr\"{o}dinger equations}, volume~10 of {\em
  Courant Lecture Notes in Mathematics}.
\newblock New York University, Courant Institute of Mathematical Sciences, New
  York; American Mathematical Society, Providence, RI, 2003.


\bibitem{Cazenave1982}
T.~Cazenave and P.~L. Lions.
\newblock Orbital stability of standing waves for some nonlinear
  {S}chr\"{o}dinger equations.
\newblock {\em Comm. Math. Phys.}, {\bf 85}(4):549--561, 1982.

\bibitem{Chang2023}
X.~J.~Chang, M.~T.~Liu, and D.~K.~Yan.
\newblock Normalized ground state solutions of nonlinear {S}chr{\"o}dinger
  equations involving exponential critical growth.
\newblock {\em J. Geom. Anal.}, {\bf 33}(3):83, 2023.

\bibitem{Chen2024}
S.~T.~Chen and X.~H.~Tang.
\newblock Another look at Schr\"odinger equations with prescribed mass.
\newblock {\em J. Differential Equations}, {\bf 386}:435--479, 2024.

\bibitem{Chen2023}
S.~T.~Chen and X.~H.~Tang.
\newblock Normalized solutions for Schr\"odinger equations with mixed dispersion and critical exponential growth in $\mathbb{R}^2$.
\newblock {\em Calc. Var. Partial Differential Equations}, {\bf 62}(9): Paper No. 261, 37 pp, 2023.

\bibitem{Cheng2016}
X.~Cheng, C.~X.~Miao and L.~F.~Zhao.
\newblock Global well-posedness and scattering for nonlinear {S}chr\"{o}dinger
  equations with combined nonlinearities in the radial case.
\newblock {\em J. Differential Equations}, {\bf 261}(6):2881--2934, 2016.

\bibitem{Cid2001}
C.~Cid and P.~Felmer.
\newblock Orbital stability of standing waves for the nonlinear
  {S}chr\"{o}dinger equation with potential.
\newblock {\em Rev. Math. Phys.}, {\bf 13}(12):1529--1546, 2001.

\bibitem{Colin2010}
M.~Colin, L.~Jeanjean and M.~Squassina.
\newblock Stability and instability results for standing waves of quasi-linear
  {S}chr{\"o}dinger equations.
\newblock {\em Nonlinearity}, {\bf 23}(6):1353, 2010.

\bibitem{Deng2021}
S.~B.~Deng, T.~X.~Hu and C.~L.~Tang.
\newblock {$N$}-{L}aplacian problems with critical double exponential
  nonlinearities.
\newblock {\em Discrete Contin. Dyn. Syst.}, {\bf 41}(2):987--1003, 2021.

\bibitem{BezerradoO1997}
J.~M.~do~\'{O}.
\newblock {$N$}-{L}aplacian equations in $\mathbb{R}^{N}$ with critical growth.
\newblock {\em Abstr. Appl. Anal.}, {\bf 2}(3-4):301--315, 1997.

\bibitem{Dou2023}
J.~B.~Dou, L.~Huang and X.~X.~Zhong.
\newblock Normalized solutions to {N}-laplacian equations in $\mathbb{R}^{N}$
  with exponential critical growth.
\newblock \textit{Preprint.}

\bibitem{Figueiredo1995}
D.~G. de~Figueiredo, O.~H. Miyagaki and B.~Ruf.
\newblock Elliptic equations in {${\bf R}^2$} with nonlinearities in the
  critical growth range.
\newblock {\em Calc. Var. Partial Differential Equations}, {\bf 3}(2):139--153, 1995.

\bibitem{Fukaya2018}
N.~Fukaya and M.~Ohta.
\newblock Strong instability of standing waves with negative energy for double
  power nonlinear {S}chr\"{o}dinger equations.
\newblock {\em SUT J. Math.}, {\bf 54}(2):131--143, 2018.

\bibitem{Giacomoni2016}
J.~Giacomoni, P.~K. Mishra and K.~Sreenadh.
\newblock Critical growth problems for {$\frac12$}-{L}aplacian in $\mathbb{R}$.
\newblock {\em Differ. Equ. Appl.}, {\bf 8}(3):295--317, 2016.

\bibitem{Gidas1981}
B.~Gidas, W.~M.~Ni and L.~Nirenberg.
\newblock Symmetry of positive solutions of nonlinear elliptic equations in
  {${\bf R}^{n}$}.
\newblock In {\em Mathematical analysis and applications, {P}art {A}}, volume~7
  of {\em Adv. Math. Suppl. Stud.}, pages 369--402. Academic Press, New
  York-London, 1981.

\bibitem{Guo2014}
Y.~J.~Guo, X.~Y.~Zeng and H.~S.~Zhou.
\newblock Concentration behavior of standing waves for almost mass critical
  nonlinear {S}chr\"{o}dinger equations.
\newblock {\em J. Differential Equations}, {\bf 256}(7):2079--2100, 2014.


\bibitem{Ikoma2019}
N.~Ikoma and K.~Tanaka.
\newblock A note on deformation argument for $L^2$ normalized solutions of
  nonlinear {S}chr{\"o}dinger equations and systems.
\newblock {\em Adv. Differential Equations}, {\bf 24}(11-12): 609--646, 2019.

\bibitem{Jeanjean1997}
L.~Jeanjean.
\newblock Existence of solutions with prescribed norm for semilinear elliptic
  equations.
\newblock {\em Nonlinear Anal.}, {\bf 28}(10):1633--1659, 1997.

\bibitem{Jeanjean2022}
L.~Jeanjean and T.~T. Le.
\newblock Multiple normalized solutions for a {S}obolev critical
  {S}chr{\"o}dinger equation.
\newblock {\em Math. Ann.}, {\bf 384}(1-2):101--134, 2022.

\bibitem{Jeanjean2020}
L.~Jeanjean and S.~S.~Lu.
\newblock A mass supercritical problem revisited.
\newblock {\em Calc. Var. Partial Differential Equations}, {\bf 59}(5):Paper No. 174, 43 pp, 2020.

\bibitem{Jeanjean2024}
L.~Jeanjean, J.~J.~Zhang and X.~X.~Zhong.
\newblock A global branch approach to normalized solutions for the
  {S}chr\"{o}dinger equation.
\newblock {\em J. Math. Pures Appl.}, {\bf 183}(9):44--75, 2024.



\bibitem{Li2008}
Y.~X.~Li and B.~Ruf.
\newblock A sharp {T}rudinger-{M}oser type inequality for unbounded domains in
  $\mathbb{R}^n$.
\newblock {\em Indiana Univ. Math. J.}, {\bf 57}(1):451--480, 2008.



\bibitem{Lions1984a}
P.-L. Lions.
\newblock The concentration-compactness principle in the calculus of
  variations. {T}he locally compact case. {I}.
\newblock {\em Ann. Inst. H. Poincar\'{e} Anal. Non Lin\'{e}aire},
  {\bf1}(2):109--145, 1984.


\bibitem{Lions1984}
P.-L. Lions.
\newblock The concentration-compactness principle in the calculus of
  variations. {T}he locally compact case. {II}.
\newblock {\em Ann. Inst. H. Poincar\'{e} Anal. Non Lin\'{e}aire},
  {\bf1}(4):223--283, 1984.


\bibitem{Liu2024}
W.~M.~Liu, X.~X.~Zhong and J.~F.~Zhou.
\newblock Normalized solution for the general {K}irchhoff type equations.
\newblock {\em Acta Math. Sci. Ser. B (Engl. Ed.)}, \textit{Preprint.} 



\bibitem{Liu2008}
Z.~L.~Liu and Z.~Q.~Wang.
\newblock Multiple bound states of nonlinear {S}chr{\"o}dinger systems.
\newblock {\em Comm. Math. Phys.}, {\bf 282}:721--731, 2008.

\bibitem{Lu2015}
G.~Z.~Lu, H.~L.~Tang, and M.~C.~Zhu.
\newblock Best constants for {A}dams' inequalities with the exact growth
  condition in $\mathbb{R}^n$.
\newblock {\em Adv. Nonlinear Stud.}, {\bf 15}(4):763--788, 2015.

\bibitem{Moser1970}
J.~Moser.
\newblock A sharp form of an inequality by {N}. {T}rudinger.
\newblock {\em Indiana Univ. Math. J.}, {\bf 20}:1077--1092, 1970/71.


\bibitem{Pohozaev1965}
\newblock S.I.~Poho\v{z}aev,
\newblock \emph{{The Sobolev embedding in the case  $pl=n$}},
\newblock Proc.Tech. Sci. Conf. on Adv. Sci., Research 1964--1965, Mathematics Section
(1965), 158--170, Moskov. \`Energet. Inst., Moscow.


\bibitem{Quittner2007}
P.~Quittner and P.~Souplet.
\newblock {\em Superlinear parabolic problems: Blow-up, global existence and
  steady states}.
\newblock Springer, 2007.

\bibitem{Radulescu2024}
V.~D.~R\u{a}dulescu, J.~J.~Zhang, X.~X.~Zhong and J.~F.~Zhou.
\newblock Prescribed mass solutions for {S}chr{\"{o}}dinger equations with critical exponents and lack of compactness.
\newblock {\em Preprint}.

\bibitem{Ruf2005}
B.~Ruf.
\newblock A sharp {T}rudinger-{M}oser type inequality for unbounded domains in
  $\mathbb{R}^2$.
\newblock {\em J. Funct. Anal.}, {\bf 219}(2):340--367, 2005.

\bibitem{Shibata2014}
M.~Shibata.
\newblock Stable standing waves of nonlinear {S}chr{\"o}dinger equations with a
  general nonlinear term.
\newblock {\em Manuscripta Math.}, {\bf 143}(1-2):221--237, 2014.

\bibitem{Sirakov2007}
B.~Sirakov.
\newblock Least energy solitary waves for a system of nonlinear
  {S}chr\"{o}dinger equations in {$\Bbb R^n$}.
\newblock {\em Comm. Math. Phys.}, {\bf271}(1):199--221, 2007.

\bibitem{Soave2020}
N.~Soave.
\newblock Normalized ground states for the {NLS} equation with combined
  nonlinearities.
\newblock {\em J. Differential Equations}, {\bf 269}(9):6941--6987, 2020.

\bibitem{Soave2020a}
N.~Soave.
\newblock Normalized ground states for the {NLS} equation with combined
  nonlinearities: The sobolev critical case.
\newblock {\em J. Funct. Anal.}, {\bf 279}(6):108610, 43 pp, 2020.

\bibitem{Stuart1982}
C.~A. Stuart.
\newblock Bifurcation for dirichlet problems without eigenvalues.
\newblock {\em Proc. Lond. Math. Soc.}, {\bf 45}(1):169--192, 1982.


\bibitem{Stuart1989}
C.~A. Stuart and K.~Kirchg{\"a}ssner.
\newblock Bifurcation from the essential spectrum for some non-compact
  non-linearities.
\newblock {\em Math. Methods Appl. Sci.}, {\bf 11}(4):525--542, 1989.


\bibitem{Talenti1976}
G.~Talenti.
\newblock Best constant in {S}obolev inequality.
\newblock {\em Ann. Mat. Pura Appl. (4)}, {\bf 110}:353--372, 1976.

\bibitem{Tao2007}
T.~Tao, M.~Visan, and X.~Y.~Zhang.
\newblock The nonlinear {S}chr\"{o}dinger equation with combined power-type
  nonlinearities.
\newblock {\em Comm. Partial Differential Equations}, {\bf 32}(7-9):1281--1343, 2007.

\bibitem{Trudinger1967}
N.~S.~Trudinger.
\newblock On imbeddings into {O}rlicz spaces and some applications.
\newblock {\em J. Math. Mech.}, {\bf 17}:473--483, 1967.

\bibitem{Wei2022}
J.~C.~Wei and Y.~Z.~Wu.
\newblock Normalized solutions for {S}chr{\"o}dinger equations with critical
  {S}obolev exponent and mixed nonlinearities.
\newblock {\em J. Funct. Anal.}, {\bf 283}(6):109574, 46 pp,2022.

\bibitem{Willem1996}
M.~Willem.
\newblock {\em Minimax Theorems}.
\newblock Minimax Theorems, 1996.

\bibitem{Zeng2023}
X.~Y.~Zeng, J.~J.~Zhang, Y.~M.~Zhang and X.~X.~Zhong.
\newblock On the {K}irchhoff equation with prescribed mass and general
  nonlinearities.
\newblock {\em Discrete Contin. Dyn. Syst. Ser. S}, {\bf 16}(11):3394--3409, 2023.



\bibitem{Zhang2022}
J.~Zhang, J.~J.~Zhang and X.~X.~Zhong.
\newblock Normalized solutions to {K}irchhoff type equations with a critical
  growth nonlinearity.
\newblock \textit{Preprint}.

\end{thebibliography}
\end{document}